\pdfoutput=1

\documentclass{siamart0216}

\usepackage{fullpage}
\usepackage{hyperref}
\usepackage{amsmath,amssymb,amsfonts,mathrsfs}
\usepackage[titletoc,toc,title]{appendix}

\usepackage{array}
\usepackage[utf8]{inputenc}
\usepackage{listings}
\usepackage{mathtools}
\usepackage{pdfpages}
\usepackage[textsize=footnotesize,color=green]{todonotes}
\usepackage{bm}
\usepackage[normalem]{ulem}
\usepackage{hhline}

\usepackage{algorithm}
\usepackage[noend]{algpseudocode}
\usepackage{algorithmicx}
\algblock{ParFor}{EndParFor}
\algnewcommand\algorithmicparfor{\textbf{parfor}}
\algnewcommand\algorithmicpardo{\textbf{do}}
\algnewcommand\algorithmicendparfor{\textbf{end\ parfor}}
\algrenewtext{ParFor}[1]{\algorithmicparfor\ #1\ \algorithmicpardo}
\algrenewtext{EndParFor}{\algorithmicendparfor}

\usepackage{graphicx}
\usepackage{subfig}
\usepackage{color}

\usepackage{pgfplots}
\usepackage{pgfplotstable}
\definecolor{markercolor}{RGB}{124.9, 255, 160.65}
\pgfplotsset{width=10cm,compat=1.3}
\pgfplotsset{
tick label style={font=\small},
label style={font=\small},
legend style={font=\small}
}

\newcommand{\td}[2]{\frac{{\rm d}#1}{{\rm d}{\rm #2}}}
\newcommand{\pd}[2]{\frac{\partial#1}{\partial#2}}

\newcommand{\nor}[1]{\left\| #1 \right\|}

\newcommand{\LRp}[1]{\left( #1 \right)}
\newcommand{\LRs}[1]{\left[ #1 \right]}

\newcommand{\LRb}[1]{\left| #1 \right|}
\newcommand{\LRc}[1]{\left\{ #1 \right\}}
\newcommand{\wip}[2]{\left( #1 \right)_{T_{#2}}}
\newcommand{\waip}[2]{\left( #1 \right)_{T^{-1}_{#2}}}
\newcommand{\wanor}[2]{\nor{#1}_{T^{-1}_{1/#2}}}

\newcommand{\Grad} {\ensuremath{\nabla}}
\newcommand{\Div} {\ensuremath{\nabla\cdot}}

\newcommand{\jump}[1] {\ensuremath{\LRs{\![#1]\!}}}
\newcommand{\avg}[1] {\ensuremath{\LRc{\!\{#1\}\!}}}

\renewcommand{\L}{L^2\LRp{\Omega}}

\newcommand{\Oh}{\Omega_h}

\newcommand{\eval}[2][\right]{\relax
  \ifx#1\right\relax \left.\fi#2#1\rvert}

\newcommand{\note}[1]{#1}

\newcolumntype{C}[1]{>{\centering\let\newline\\\arraybackslash\hspace{0pt}}m{#1}}

\newcommand*\diff[1]{\mathop{}\!{\mathrm{d}#1}}
\makeatletter
\renewcommand\d[1]{\mspace{6mu}\mathrm{d}#1\@ifnextchar\d{\mspace{-3mu}}{}}
\makeatother

\author{Jesse Chan\thanks{Department of Mathematics, Virginia Tech, Blacksburg, VA} \and Russell J.\ Hewett\thanks{TOTAL E\&P Research and Technology USA, Houston, TX} \and T.\ Warburton\footnotemark[1]}
\date{}
\title{Weight-adjusted discontinuous Galerkin methods: wave propagation in heterogeneous media}

\begin{document}

\maketitle

\begin{abstract}
Time-domain discontinuous Galerkin (DG) methods for wave propagation require accounting for the inversion of dense elemental mass matrices, where each mass matrix is computed with respect to a parameter-weighted $L^2$ inner product.  In applications where the wavespeed varies spatially at a sub-element scale, these matrices are distinct over each element, necessitating additional storage.  In this work, we propose a weight-adjusted DG (WADG) method which reduces storage costs by replacing the weighted $L^2$ inner product with a weight-adjusted inner product.  This equivalent inner product results in an energy stable method, but does not increase storage costs for locally varying weights.  \textit{A priori} error estimates are derived, and numerical examples are given illustrating the application of this method to the acoustic wave equation with heterogeneous wavespeed.  
\end{abstract}



\section{Introduction}

Accurate numerical simulations of wave propagation through complex media are becoming increasingly important in seismology, especially as modern computational resources make the use of high fidelity subsurface models feasible for seismic imaging and full waveform inversion.  A host of different numerical methods are currently in use, the most popular of which are high order finite difference methods \cite{virieux2011review}.  While finite difference methods tend to perform excellently for simple geometries and smoothly varying data, their accuracy is degraded for heterogeneous media with interfaces or sharp gradients \cite{symes2009interface}.  

In order to address these issues, high order finite element methods for wave propagation have been considered as alternatives to finite difference methods.  A drawback of using continuous finite elements for time-domain simulations using explicit timestepping is the inversion of a global mass matrix system at each timestep.  Spectral Element Methods (SEM) \cite{komatitsch1998spectral} address this issue by diagonalizing this mass matrix system through the use of mass-lumping, which co-locates interpolation nodes for Lagrange basis functions and Gauss-Legendre-Lobatto quadrature points.  Since SEM is limited to unstructured hexahedral meshes, which are less geometrically flexible than tetrahedral meshes, triangular and tetrahedral mass-lumped spectral element methods have been investigated as alternatives \cite{chin1999higher, cohen2001higher, zhebel2014comparison}.  However, due to a mismatch in the number of natural quadrature nodes and the dimension of polynomial approximation spaces on simplices, these methods necessitate adding additional nodes in the interior of the element to construct  sufficiently accurate nodal points suitable for mass-lumping.  Additionally, mass-lumpable nodal points on tetrahedra have only been determined for polynomial bases of degree four or less \cite{chin1999higher}.  

High order discontinuous Galerkin (DG) methods have been considered as an alternative to Spectral Element Methods for seismic wave propagation \cite{dumbser2006arbitrary, dumbser2007arbitrary, de2008interior, etienne2010hp}.  Instead of using mass-lumping to arrive at a diagonal mass matrix, DG methods naturally induce a block diagonal mass matrix through the use of arbitrary-order approximation spaces which are discontinuous across element boundaries. Weak continuity of approximate solutions in such spaces is enforced through numerical fluxes on shared faces.  The local nature and fixed communication patterns DG methods also makes them well-suited for parallelization, and the scalability of DG methods for time-domain wave propagation problems has been demonstrated for hundreds of thousands of cores \cite{wilcox2010high}.  Additionally, the computational structure of DG methods has been shown to be well-suited to many-core and accelerator architectures such as Graphics Processing Units (GPU).  DG implementations on a single GPU have demonstrated significant speedups over conventional architectures \cite{klockner2009nodal,fuhry2014discontinuous}, while implementations using multiple GPUs still demonstrate high scalability \cite{godel2010scalability,modave2015accelerated}.  

A limitation of many implementations of DG is that the wavespeed is assumed to be piecewise constant over each element, which can lead to spurious reflections and loss of high order accuracy.  In order to accomodate locally heterogeneous models over each element,  Castro et al.\ discretize a pseudo-conservative form of the wave equation \cite{castro2010seismic}.  However, this requires including additional source terms to account for local spatial variation of material parameters, which makes it difficult to prove energy stability or high order accuracy.  An alternative approach was taken by Mercerat and Glinsky in \cite{mercerat2015nodal}, where the spatial variation of the wavespeed is incorporated into local elemental mass matrices as a weighting function.  This approach can be shown to be energy stable; however, since the wavespeed can vary from element to element, this necessitates either expensive on-the-fly solutions of dense matrix equations or the storage of factorizations/inverses of local mass matrices.  This presents a challenge for GPU implementations, as the former is computationally expensive and not well-suited to the fine-grain parallelism of GPUs, while the latter greatly increases storage costs for high order approximations.  Storage costs are especially problematic for GPU implementations of DG, due to limited global memory on accelerator architectures.  Efficient implementations have also typically relied on the fact that, for affinely mapped tetrahedra and triangles, each block of the mass matrix is identical up to a constant scaling of a single reference mass matrix.  Additionally, since GPUs require sufficiently large problem sizes for peak efficiency, increased storage costs can decrease the efficiency of GPU-based implementations.  

Since similar storage issues are encountered for DG methods on non-affine elements, the Low-Storage Curvilinear DG (LSC-DG) method was introduced in \cite{warburton2010low,warburton2013low} to reduce the asymptotic storage costs for high order DG methods on curvilinear meshes by incorporating locally varying geometric factors into the basis functions on each element.  When coupled with an \textit{a priori} stable quadrature-based variational formulation, the LSC-DG method can be shown to be both energy stable and high order accurate.  It is straightforward to adapt LSC-DG to reduce storage costs for DG in the presence of heterogeneous wavespeeds; however, doing so forfeits the computational advantages available under specific choices of basis, such as nodal or Bernstein-Bezier polynomials \cite{hesthaven2007nodal, chan2015bbdg}.  

This work addresses these issues by introducing a weight-adjusted DG (WADG) method for heterogeneous media.  In particular, the weight-adjusted DG method is energy stable and high order convergent, while maintaining much of the computational structure of existing DG methods for isotropic media.  The techniques in this work resemble those used in quadrature-free DG methods for hyperbolic problems \cite{atkins1998quadrature}, though the implementations presented in this work still rely explicitly on quadrature for a low-storage implementation.  The main idea of the WADG method is to replace the weighted mass matrices of Mercerat and Glinsky \cite{mercerat2015nodal} with an equivalent weight-adjusted mass matrix which yields a low-storage inversion.  The structure of this paper is as follows: Section~\ref{sec:form} introduces standard DG methods for wave propagation in heterogeneous media based on the use of  weighted $L^2$ inner products \cite{mercerat2015nodal}.  Section~\ref{sec:ip} introduces operators used to define an equivalent weight-adjusted inner product, and Section~\ref{sec:wadg} introduces the weight-adjusted DG method, along with discussions of local conservation and an \textit{a priori} error analysis.  Finally, Section~\ref{sec:num} provides numerical experiments which validate theoretical estimates.  


\section{Mathematical notation}
\label{sec:notation}
We begin with the assumption that the domain $\Omega$ is Lipschitz, and is represented exactly by a triangulation $\Omega_h$ consisting of elements $D^k$, where each element is the image of a reference element under the elemental mapping 
\[
\bm{x}^k = \bm{\Phi}^k \widehat{\bm{x}},
\]
where $\bm{x}^k = \LRc{x^k,y^k,z^k}$ are physical coordinates on the $k$th element and $\widehat{\bm{x}} = \LRc{\widehat{x},\widehat{y},\widehat{z}}$ are coordinates on the reference element.  We denote the Jacobian of the transformation for the element $D^k$ as $J^k$.  

Over each element $D^k \in \Oh$, the approximation space $V_h\LRp{D^k}$ is defined as
\[
V_h\LRp{D^k} = \bm{\Phi}^k \circ V_h\LRp{\widehat{D}}.
\]
where $V_h\LRp{\widehat{D}}$ is an approximation space over the reference element.  In this work, $\widehat{D}$ is taken to be the reference bi-unit triangle or tetrahedron, while $V_h\LRp{\widehat{D}}$ is taken to be the space of total degree $N$ polynomials on the reference triangle
\[
V_h\LRp{\widehat{D}} = P^N\LRp{\widehat{D}} = \LRc{ \widehat{x}^i \widehat{y}^j, \quad 0 \leq i + j \leq N}.
\]
or on the reference tetrahedron
\[
V_h\LRp{\widehat{D}} = P^N\LRp{\widehat{D}} = \LRc{ \widehat{x}^i \widehat{y}^j \widehat{z}^k, \quad 0 \leq i + j + k \leq N}.
\]
However, the analysis and methods are readily extendible to other affinely mapped element types and approximation spaces, such as tensor product degree $N$ polynomials on quadrilaterals and hexahedra.  The global approximation space is taken to be the direct sum of approximation spaces over each element
\[
V_h\LRp{\Omega_h} = \bigoplus_{D^k} V_h\LRp{D^k}.
\]
We define $\Pi_N$ as the $L^2$ projection onto $P^N\LRp{D^k}$ such that
\[
\LRp{\Pi_N u,v}_{L^2\LRp{D^k}} = \LRp{u,v}_{L^2\LRp{D^k}}, \qquad v\in P^N\LRp{D^k},
\]
where $\LRp{\cdot,\cdot}_{L^2\LRp{D^k}}$ denotes the $L^2$ inner product over $D^k$. 

We also introduce the standard Lebesgue $L^p$ norms over a general domain $\Omega$
\begin{align*}
\nor{u}_{L^p\LRp{\Omega}} &= \LRp{\int_{\Omega} u^p}^{1/p} \qquad 1 \leq p < \infty \\
\nor{u}_{L^{\infty}\LRp{\Omega}} &= \inf\LRc{C \geq 0: \LRb{u\LRp{\bm{x}}} \leq C \quad \forall \bm{x}\in \Omega},
\end{align*}
and the associated $L^p$ spaces
\begin{align*}
L^p\LRp{\Omega} &= \LRc{u: \Omega\rightarrow \mathbb{R}, \quad \nor{u}_{L^p\LRp{\Omega}} < \infty} \qquad 1\leq p < \infty \\
L^{\infty}\LRp{\Omega} &= \LRc{u: \Omega\rightarrow \mathbb{R}, \quad \nor{u}_{L^{\infty}\LRp{\Omega}} < \infty}.
\end{align*}
The $L^p$ Sobolev seminorms and norms of degree $s$ are then defined 
\begin{align*}
\LRb{u}_{W^{s,p}\LRp{\Omega}} &= \LRp{\sum_{\LRb{\alpha}= s} \nor{ D^{\alpha} u}_{L^p\LRp{\Omega}}^p}^{1/p}, \qquad \LRb{u}_{W^{s,\infty}\LRp{\Omega}} = \max_{\LRb{\alpha}= s} \nor{D^{\alpha}u}_{L^{\infty}\LRp{\Omega}}\\
\nor{u}_{W^{s,p}\LRp{\Omega}} &= \LRp{\sum_{\LRb{\alpha}\leq s} \nor{ D^{\alpha} u}_{L^p\LRp{\Omega}}^p}^{1/p}, \qquad \nor{u}_{W^{s,\infty}\LRp{\Omega}} = \max_{\LRb{\alpha}\leq s} \nor{D^{\alpha}u}_{L^{\infty}\LRp{\Omega}}.
\end{align*}
where $\alpha = \LRc{\alpha_1,\alpha_2,\alpha_3}$ is a multi-index such that
\[
D^{\alpha}u = \pd{^{\alpha_1}}{x^{\alpha_1}}\pd{^{\alpha_2}}{y^{\alpha_2}}\pd{^{\alpha_3}}{z^{\alpha_3}} u,
\]

\section{Discontinuous Galerkin methods for the acoustic wave equation}
\label{sec:form}

We introduce the jump and average of $u\in V_h\LRp{\Omega_h}$ as follows: let $f$ be a shared face between two elements $D^{k^-}$ and $D^{k^+}$, and let $u$ and $\bm{u}$ be scalar and vector valued functions, respectively.  The jumps and averages of $u, \bm{u}$ are defined as
\[
\jump{u} = u^+ - u^-, \qquad \avg{u} = \frac{u^+ + u^-}{2}, \qquad \jump{\bm{u}} = \bm{u}^+ - \bm{u}^-, \qquad \avg{\bm{u}} = \frac{\bm{u}^+ + \bm{u}^-}{2}.
\]

In this work, we consider the acoustic wave equation as a model problem.  In first order form, this is given by
\begin{align*}
\frac{1}{\rho c^2}\pd{p}{t}{} + \Div \bm{u} &= 0,\\
\rho\pd{\bm{u}}{t}{} + \Grad p &= 0,
\end{align*}
where $t$ is time, $p$ is pressure, $\bm{u}$ is the vector velocity, and $\rho$ and $c^2$ are density and wavespeed, respectively.  \note{
We will assume that $c^2$ is bounded from above and below
\[
0 < c_{\min}\leq c^2(\bm{x})\leq c_{\max} < \infty.
\]
}
We adopt the discontinuous Galerkin variational formulation of \cite{warburton2013low}, which is given over each element $D^k$ by
\begin{align}
\int_{D^k} \frac{1}{\rho c^2}\pd{p}{t}{}v \diff x &= -\int_{D^k} \Div\bm{u}v \diff x + \int_{\partial D^k} \frac{1}{2}\LRp{\tau_p\jump{p} - \bm{n}\cdot \jump{\bm{u}} }v^- \diff x,  \nonumber \\ 
\int_{D^k} \rho\pd{\bm{u}}{t}{}\bm{\tau} \diff x &= - \int_{D^k} \Grad p \cdot  \bm{\tau} \diff x + \int_{\partial D^k} \frac{1}{2}\LRp{\tau_u \jump{\bm{u}}\cdot \bm{n}^- - \jump{p}}\bm{\tau}^-\cdot \bm{n}^- \diff x.
\label{eq:form}
\end{align}
where $\bm{n}$ is the outward unit normal vector, $\tau_p = 1/\avg{\rho c}$, and $\tau_u = \avg{\rho c}$.  We refer to this DG method as the standard DG method for the remainder of this work.  Finally, we note that the weight-adjusted DG method proposed in this paper impacts only the computation of mass matrices, and thus is not tied to a single choice of DG formulation or numerical flux.  

The formulation (~\ref{eq:form}) can be shown to be energy stable for any choice of $\tau_p, \tau_u \geq 0$ \cite{warburton2013low}, and the specific choice of $\tau_p, \tau_u$ reduce the numerical flux to the upwind fluxes (as determined by the solution of a Riemann problem) for constant $\rho, c$.  For the remainder of this work, we assume $\rho = 1$ for simplicity, though it is straightforward to adapt the results to non-constant $\rho$.  

Finally, for this work, we assume homogeneous Dirichlet boundary conditions $p=0$ on $\partial \Omega$.  These are enforced through reflection conditions at boundary faces $f \in \partial \Omega$
\[
\left.p^+\right|_{f} = -\left.p^-\right|_{f}, \qquad \left.\bm{n}^+\bm{u}^+\right|_{f} = \left.\bm{n}^-\bm{u}^-\right|_{f}.
\]

\subsection{Discrete formulation}

Assuming that $V_h\LRp{\widehat{D}}$ is spanned by the basis $\LRc{\phi_i}_{i=1}^{N_p}$, the discrete formulation of the DG method is given most simply in terms of mass, (weak) differentiation, and lift matrices.  The mass matrix $\bm{M}^k$, weighted mass matrix $\bm{M}_{1/c^2}^k$ and face mass matrix $\bm{M}^{k}_f$ for the element $D^k$ are defined as
\begin{align*}
\LRp{\bm{M}^k}_{ij} &= \int_{D^k} \phi_j \phi_i = \int_{\widehat{D}}{ \phi_j \phi_i} J^k,\\
\LRp{\bm{M}^k_{1/c^2}}_{ij} &= \int_{D^k} \frac{1}{c^2}\phi_j \phi_i = \int_{\widehat{D}}{ \frac{1}{c^2}\phi_j \phi_i} J^k,\\
\LRp{\bm{M}^{k}_f}_{ij} &= \int_{\partial D^{k}_f} \phi_j \phi_i = \int_{\widehat{D}_f}  \phi_j \phi_i J^{k}_f.
\end{align*}
where $J^{k}_f$ is the Jacobian of the mapping from the face of a reference element $\widehat{D}_f$ to the face of a physical element $D^{k}_f$. We also define weak differentiation matrices $\bm{S}_x, \bm{S}_y, \bm{S}_z$ with entries
\begin{align*}
\LRp{\bm{S}_x}_{ij} = \int_{\widehat{D}} \pd{\phi_j}{x} \phi_i J^k, \qquad \LRp{\bm{S}_y}_{ij} = \int_{\widehat{D}} \pd{\phi_j}{y} \phi_i J^k, \qquad \LRp{\bm{S}_z}_{ij} = \int_{\widehat{D}} \pd{\phi_j}{z} \phi_i J^k.
\end{align*}
The discrete standard DG formulation is then given in terms of these matrices.  For succinctness, we relabel subscripts $x,y,z$ as $1,2,3$ such that
\[
\LRc{\bm{S}^k_x, \bm{S}^k_y, \bm{S}^k_z} = \LRc{\bm{S}^k_1, \bm{S}^k_2, \bm{S}^k_3}, \qquad \bm{n} = \LRc{n_x, n_y, n_z} = \LRc{n_1, n_2, n_3}
\]
Then, the discrete formulation is
\begin{align*}
\bm{M}_{w}^k\td{\bm{p}}{t} &= -\sum_{j = 1,2,3}\bm{S}_{j}^k \bm{U}_j + \sum_{f=1}^{N_{\text{faces}}}\bm{M}^k_f F_p(\bm{p}^-,\bm{p}^+,\bm{U}^-,\bm{U}^+),\\
\bm{M}^k\td{\bm{U}_i}{t} &= -\bm{S}_{i}^k \bm{p} + \sum_{f=1}^{N_{\text{faces}}}  {n}_{i}\bm{M}^k_f  F_{u}(\bm{p}^-,\bm{p}^+,\bm{U}^-,\bm{U}^+), \qquad i = 1,2,3.
\end{align*}
where $w = 1/c^2$, $\bm{U}_i$ and $\bm{p}$ are degrees of freedom for $\bm{u}_i$ and $p$, and the superscripts $+$ and $-$ indicate degrees of freedom for functions on $D^k$ and its neighbor across face $f$.  $F_p,F_u$ are defined such that 
\begin{align*}
\LRp{ \bm{M}^k_f F_p(\bm{p}^-,\bm{p}^+,\bm{U}^-,\bm{U}^+)}_i &= \int_{f_{D^k}} \frac{1}{2}\LRp{\tau_p \jump{p} - \bm{n}^-\cdot\jump{\bm{u}}}\phi_i^-,\\ 
\LRp{ \bm{n}_i \bm{M}^k_f F_u(\bm{p}^-,\bm{p}^+,\bm{U}^-,\bm{U}^+)}_i &= \int_{f_{D^k}} \frac{1}{2}\LRp{\tau_u\jump{\bm{u}} \cdot \bm{n}^- - \jump{p}}\bm{\tau}_i^- \bm{n}_i^-.
\end{align*}
Inverting $\bm{M}^k_{1/c^2},\bm{M}^k$ produces a system of ODEs which can be solved using standard time-integration techniques.  

\subsection{Energy stability in a weighted $L^2$ norm}
\label{sec:energy}
When the wavespeed $1/c^2$ is incorporated into the mass matrix, it is straightforward to show that the discrete DG formulation is energy stable (in the sense that an appropriate norm of the solution is dissipative in time).  This can be shown by taking $v = p, \bm{\tau} = \bm{u}$ in the local DG formulation.  Integrating the divergence term of the pressure equation by parts gives
\begin{align*}
\int_{D^k} \frac{1}{ c^2}\pd{p}{t}{}p \diff x &= \int_{D^k} \bm{u}\Grad p \diff x + \int_{\partial D^k} \LRp{\frac{\tau_p}{2}\jump{p} - \bm{n}\cdot \avg{\bm{u}} }p \diff x, \nonumber \\ 
\int_{D^k} \pd{\bm{u}}{t}{}\bm{u} \diff x &= - \int_{D^k} \Grad p \cdot  \bm{u} \diff x + \int_{\partial D^k} \frac{1}{2}\LRp{\tau_u \jump{\bm{u}}\cdot \bm{n}^- - \jump{p}}\bm{u}\cdot \bm{n}^- \diff x.
\end{align*}
Then, adding the pressure and velocity equations together and summing over all elements $D^k$ gives 
\begin{align}
\pd{}{t}\sum_{k}\int_{D^k} \frac{1}{c^2}p^2   + \LRb{\bm{u}}^2 = \pd{}{t}\sum_{k} \wip{p,p}{1/c^2} + \LRp{\bm{u},\bm{u}} = - \sum_{k}\frac{1}{2}\int_{\partial D^k}  \tau_p \jump{p}^2 + \tau_u \LRp{\bm{n}\cdot\jump{\bm{u}}}^2 < 0.
\label{eq:stability}
\end{align}
where we have introduced the weighted $L^2$ inner product over $D^k$
\[
\LRp{wp,v}_{L^2\LRp{D^k}} = \int_{D^k} w p v.
\]
Assuming that the wavespeed is bounded from above and below by $0 < c_{\min} \leq c \leq c_{\max} < \infty$, the quantity
\begin{align}
\sum_{k} \LRp{\frac{p}{c^2},p}_{L^2\LRp{D^k}} + \LRp{\bm{u},\bm{u}} 
\label{eq:weightedL2}
\end{align}
defines a squared norm on $\LRp{p,\bm{u}}$, and (\ref{eq:stability}) implies that this weighted $L^2$ norm of the solution is non-increasing in time.  Thus, incorporating wavespeed into the left hand side of the DG formulation (and into the mass matrices of the discrete formulation) results in an energy stable method.  This approach is taken by Mercerat and Glinsky \cite{mercerat2015nodal} to develop a nodal DG method for elastic wave propagation in heterogeneous media.  However, this also greatly increases storage costs if $c$ varies locally over each element. 

Consider the case when all elements $D^k$ are planar simplices (implying that the mapping $\bm{\Phi}^k$ is affine and $J^k$ is constant) and $c$ is piecewise constant over each element $D^k$.  Then, the mass matrices $\bm{M}^k_{1/c^2}, \bm{M}^k$ satisfy
\[
\bm{M}^k_{1/c^2} = \frac{1}{c^2}J^k \widehat{\bm{M}}, \qquad \bm{M}^k = J^k \widehat{\bm{M}}.
\]
Under these assumptions, all mass matrices are simply scalings of the reference mass matrix.  Inversion of the mass matrix can be dealt with by pre-multiplying reference matrices by the inverse of the reference mass matrix \cite{hesthaven2007nodal}.  However, when $c$ varies locally over an element, each mass matrix is distinct, requiring either iterative solvers or storage of dense matrices/factorizations to apply the inverse.  

Several approaches can be taken to address these storage costs.  Castro et al.\ \cite{castro2010seismic} multiply the pressure equation on both sides by $c^2$ to remove the variation of $c$ from the mass matrix.  However, this rewrites the wave equation in non-conservative form of the wave equation, which does not lend itself readily to an energy stable DG formulation.  Castro et al.\ introduce new source terms into the formulation to overcome this difficulty, rewriting the wave equation in a pseudo-conservative form.  However, it is not obvious whether this formulation is energy stable.  It is also possible to build the variation of $c$ into the basis, as is done with spatially varying Jacobian factors $J^k$ for non-affine elements in \cite{warburton2013low}.  However, this introduces rational basis functions, which require explicit quadrature-based \textit{a priori} stable variational formulations for energy stability.  We propose an alternative approach in this work, which allows for the use of polynomial basis functions while maintaining a low storage implementation based on a weight-adjusted inner product.  

\section{Approximating weighted $L^2$ inner products}
\label{sec:ip}

In order to introduce the new DG method, we introduce a new inner product under which the proposed method is energy stable.  The construction of this inner product is based on operators $T_w, T^{-1}_w$ which approximate polynomial multiplication and division by a weight $w$, respectively.  Intuitively, this inner product approximates the weighted $L^2$ inner product (\ref{eq:weightedL2}) under which the DG method is shown to be energy stable in Section~\ref{sec:energy}.

\subsection{Approximating polynomial multiplication and division}

Let $w(\bm{x})$ be a scalar weight defined on the domain $\Omega$ which is bounded from above and below 
\[
0 < w_{\min} \leq w \leq w_{\max} < \infty.
\]
We define the operator $T_w: L^2\LRp{D^k} \rightarrow P^N\LRp{D^k}$ 
\[
T_w u = \Pi_N\LRp{wu}.
\]
Since $T_w$ also satisfies
\begin{align*}
\LRp{T_wu,v}_{L^2\LRp{D^k}} &= \LRp{\Pi_N(wu),v}_{L^2\LRp{D^k}}= \LRp{u,wv}_{L^2\LRp{D^k}}\\
&= \LRp{u,\Pi_N\LRp{wv}}_{L^2\LRp{D^k}} = \LRp{u,T_wv}_{L^2\LRp{D^k}},
\end{align*}
it is self-adjoint and positive definite, and induces a weighted inner product $\wip{\cdot,\cdot}{w}$ over $D^k$
\[
\wip{u,v}{w} \coloneqq \LRp{wu,v}_{L^2\LRp{D^k}}.
\]
For $u,v \in P^N\LRp{D^k}$, this weighted inner product reduces to the weighted $L^2$ inner product 
\[
\wip{u,v}{w} = \LRp{T_wu,v}_{L^2\LRp{D^k}}  = \LRp{\Pi_N\LRp{wu},v}_{L^2\LRp{D^k}}= \LRp{wu,v}_{L^2\LRp{D^k}}, \qquad u,v \in P^N\LRp{D^k} .
\]
We also define an operator $T_w^{-1}$ as
\[
T_w^{-1}: L^2\LRp{D^k} \rightarrow P^N\LRp{D^k}, \qquad \LRp{w T_w^{-1}u,v}_{{D}^k} = \LRp{u,v}_{{D}^k}.
\]
$T_w^{-1}$ can be considered the inverse of $T_w$ in the following sense:
\begin{lemma}
$T_w^{-1}T_w  = T_w T_w^{-1} = \Pi_N.$
\label{lemma:properT1}
\end{lemma}
\begin{proof}
By the definitions of $T_w, T_w^{-1}$, 
\begin{align*}
\LRp{T_w T_w^{-1} u,v}_{L^2\LRp{D^k}} &=\LRp{ wT_w^{-1} u,v}_{L^2\LRp{D^k}} = \LRp{u,v}_{L^2\LRp{D^k}}, \qquad \forall v\in P^N\LRp{D^k},\\
\LRp{w T_w^{-1} T_w u,v}_{L^2\LRp{D^k}} &=\LRp{ \Pi_N \LRp{wu},v}_{L^2\LRp{D^k}} = \LRp{ wu,v}_{L^2\LRp{D^k}}, \qquad \forall v\in P^N\LRp{D^k}.
\end{align*}
These implies that when the domain of $T_w$ is restricted to $P^N\LRp{D^k}$, $T_w^{-1}$ satisfies $T_w^{-1}T_w = T_wT_w^{-1} = I$. More generally, when the domain of $T_w, T_w^{-1}$ is $L^2\LRp{D^k}$, 
\[
T_w^{-1}T_w = T_wT_w^{-1} = \Pi_N.
\] 
\end{proof}

We also have the following properties of the operator $T_w^{-1}$
\begin{lemma}
The weighted operator $T^{-1}_w$ satisfies
\[
\Pi_N T^{-1}_w = T^{-1}_w \Pi_N = T^{-1}_w, \qquad \nor{T^{-1}_w u }_{L^2} \leq \frac{1}{w_{\min}} \nor{u}_{L^2}.
\]
\label{lemma:properT}
\end{lemma}
\begin{proof}
The first equality is simply because $T^{-1}_w u \in P^N\LRp{D^k}$ and $\Pi_N$ restricted to $P^N\LRp{D^k}$ is the identity map.  The second equality is verified by using the definition of $T^{-1}_w, \Pi_N$ and showing that 
\[
\LRp{wT^{-1}_w \Pi_N u,v}_{D^k} = \LRp{\Pi_N u,v}_{D^k} = \LRp{u,v}_{D^k} = \LRp{wT^{-1}_w u,v}_{D^k}.
\]
The norm of $\nor{T^{-1}_wu}_{L^2}$ can be bounded by noting
\begin{align*}
\nor{T^{-1}_w u}_{L^2} &\leq \LRp{\frac{1}{w_{\min}} \LRp{wT^{-1}_w u,T^{-1}_w u}}^{1/2} \leq\LRp{\frac{1}{w_{\min}} \LRp{u,T^{-1}_w u}}^{1/2},\\
& \leq \LRp{\frac{1}{w^2_{\min}} \LRp{u, wT^{-1}_w u}}^{1/2} = \LRp{\frac{1}{w^2_{\min}} \LRp{u, u}}^{1/2}.
\end{align*}
\end{proof}
This also implies that $\nor{T_w^{-1}}_{L^2\LRp{D^k}} \leq \frac{1}{w_{\min}}$.

\subsection{A weight-adjusted inner product}

The introduction of the weight-adjusted DG method relies an approximation of the weighted $L^2$ inner product 
\[
\LRp{wu,v}_{L^2\LRp{D^k}} = \LRp{T_w u,v}_{L^2\LRp{D^k}} = \wip{u,v}{w}
\]
by an equivalent inner product, based on the observation that
\[
T_{w}u \approx T^{-1}_{1/w}u.
\]
In other words (for appropriate weighting functions $w$) the projected multiplication operator $T_w$ is well-approximated by the inverse of the projected polynomial division operator $T^{-1}_{1/w}$.  This weight ``adjustment'' will make it possible to approximate the inverse of the weighted mass matrix in a low-storage, matrix-free manner.

We introduce the map $\waip{\cdot,\cdot}{1/w}: L^2\LRp{D^k} \times L^2\LRp{D^k} \rightarrow \mathbb{R}$ using $T_{1/w}^{-1}$
\[
\waip{u,v}{1/w} \coloneqq \LRp{T^{-1}_{1/w} u,v}_{{D}^k}.
\]
For positive weight function $w$, this map defines an inner product, which we refer to as the weight-adjusted inner product:
\begin{lemma}
\label{lemma:equiv}
$\waip{u,v}{1/w}$ defines an inner product on $P^N\LRp{D^k}\times P^N\LRp{D^k}$ with induced norm $\wanor{u}{w}$.  Additionally, $\wanor{u}{w}$ is equivalent to the $L^2$ norm over ${D}^k$ with equivalence constants
\[
{\sqrt{w_{\min}}} \nor{u}_{L^2\LRp{{D}^k}} \leq \wanor{u}{w} \leq {\sqrt{w_{\max}}} \nor{u}_{L^2\LRp{{D}^k}}.
\]
\end{lemma}
\begin{proof}
It is straightforward to show that $\waip{u,v}{1/w}$ is bilinear.  Symmetry follows from the self-adjoint nature of $T_{1/w}$ and Lemma~\ref{lemma:properT1}
\begin{align*}
\waip{u,v}{1/w} &= \LRp{T_{1/w}^{-1} u, v}_{L^2\LRp{D^k}} = \LRp{T_{1/w}^{-1} u,T_{1/w} T_{1/w}^{-1}v}_{L^2\LRp{D^k}} = \LRp{u,T_{1/w}^{-1} v}_{L^2\LRp{D^k}},
\end{align*}
while positive definiteness is a result of
\[
\waip{u,u}{1/w} = \LRp{T^{-1}_{1/w} u,u}_{L^2\LRp{D^k}} \geq {w_{\min}} \LRp{ \frac{1}{w}T^{-1}_{1/w} u,u}_{L^2\LRp{{D}^k}} = {w_{\min}} \LRp{u,u}_{L^2\LRp{{D}^k}}.
\]
To show equivalence of the norm, all that remains is showing the upper bound
\begin{align*}
\wanor{u}{w}^2 &= \LRp{T^{-1}_{1/w} u,u}_{L^2\LRp{D^k}} = \LRp{\frac{1}{w} w T^{-1}_{1/w} u,u}_{L^2\LRp{{D}^k}}\\
&\leq {w_{\max}} \LRp{\frac{1}{w}T^{-1}_{1/w} u,u}_{L^2\LRp{{D}^k}} = {w_{\max}} \LRp{u,u}_{L^2\LRp{{D}^k}}.
\end{align*}
\end{proof}
For $w$ constant, $\waip{u,v}{1/w}$ reduces to a scaling of the standard $L^2$ inner product by $w$.  

We also note that the equivalence constants in this case are the same as for the weighted $L^2$ inner product $\wip{\cdot,\cdot}{w}$ over $P^N\LRp{D^k} \times P^N\LRp{D^k}$
\[
{\sqrt{w_{\min}}} \nor{u}_{L^2\LRp{{D}^k}} \leq \sqrt{\LRp{{w}u,u}_{L^2\LRp{{D}^k}}} = \sqrt{\wip{u,u}{w}} \leq  {\sqrt{w_{\max}}} \nor{u}_{L^2\LRp{{D}^k}},
\]
which appears in the standard DG formulation for spatially varying wavespeed.

\subsection{Estimates for $T_{w}, T_{1/w}^{-1}$, and $\waip{\cdot,\cdot}{1/w}$}
\label{sec:estimates}

Intuitively, both $T_{w}u$ and $T^{-1}_{1/w} u$ approximate ${w} u$, and we can quantify the accuracy of this approximation by bounding $\nor{{u}{w}-T_{w} u}_{{D}^k}$ and $\nor{{u}{w}-T^{-1}_{1/w} u}_{{D}^k}$ for weights $w$ which are sufficiently regular.  These regularity requirements are made explicit using Sobolev norms introduced in Section~\ref{sec:notation}.  

To bound the difference between $uw$ and $T_wu, T^{-1}_{1/w} u$, we require the standard interpolation estimate
\begin{align*}
\nor{u-\Pi_N u}_{{D}^k} &\leq Ch^{N+1} \nor{u}_{W^{N+1,2}\LRp{{D}^k}},
\end{align*}
which assumes $u \in W^{N+1,2}\LRp{D^k}$ and follows from the Bramble-Hilbert lemma and a scaling assumption \cite{brenner2007mathematical, warburton2013low}.

We also make use of an estimate for a weighted projection, adapted from Theorem 3.1 in \cite{warburton2013low} for an affinely mapped element:
\begin{theorem}
\label{thm:wproj}
Let $D^k$ be a quasi-regular element with representative size $h = {\rm diam}\LRp{D^k}$.  For $N \geq 0$, $w\in W^{N+1,\infty}\LRp{D^k}$, and $u\in W^{N+1,2}\LRp{D^k}$, 
\[
\nor{u - \frac{1}{w} \Pi_N\LRp{{u}{w}}}_{L^2\LRp{D^k}} \leq C h^{N+1}\nor{\frac{1}{w}}_{L^{\infty}\LRp{D^k}} \nor{w}_{W^{N+1,\infty}\LRp{D^k}}  \nor{u}_{W^{N+1,2}\LRp{D^k}}.
\]
\end{theorem}
\begin{proof}
By the Bramble-Hilbert lemma \cite{brenner2007mathematical},
\begin{align*}
\nor{u - \frac{1}{w} \Pi_N\LRp{{u}{w}}}_{L^2\LRp{D^k}} &\leq \sqrt{J^k}\nor{\frac{1}{w}}_{L^{\infty}\LRp{\widehat{D}}} \nor{uw - \Pi_N \LRp{uw}}_{L^2\LRp{\widehat{D}}} \\
&\leq  \sqrt{J^k}\nor{\frac{1}{w}}_{L^{\infty}\LRp{\widehat{D}}} \LRb{uw}_{W^{N+1,2}\LRp{\widehat{D}}}.
\end{align*}
For quasi-regular elements, a scaling argument gives
\[
\LRb{uw}_{W^{N+1,2}\LRp{\widehat{D}}} \leq C_1 h^{N+1} \frac{1}{\sqrt{J^k}} \nor{uw}_{W^{N+1,2}\LRp{D^k}}.
\]
Finally, the Sobolev norm of $uw$ may be bounded by the product of the norms of $u,w$ using the Leibniz product rule and H\"older's inequality \cite{burenkov1998sobolev}
\[
\nor{uw}_{W^{N+1,2}\LRp{D^k}} \leq C_2 \nor{w}_{W^{N+1,\infty}\LRp{D^k}}\nor{u}_{W^{N+1,2}\LRp{D^k}}.
\]
Combining these gives the desired bound.  
\end{proof}

We can now prove the following bounds:
\begin{theorem}
Let $D^k$ be a quasi-regular element with representative size $h = {\rm diam}\LRp{D^k}$.  For $N \geq 0$, $w\in W^{N+1,\infty}\LRp{D^k}$, and $u\in W^{N+1,2}\LRp{D^k}$, 
\begin{align}
\nor{{u}{w}-T_{w}u}_{L^2\LRp{D^k}} &\leq C_wh^{N+1} \nor{u}_{W^{N+1,2}\LRp{D^k}},\\
\nor{{u}{w}-T^{-1}_{1/w}u}_{L^2\LRp{D^k}} &\leq C_wh^{N+1}   \nor{u}_{W^{N+1,2}\LRp{D^k}}.
\end{align}
where $C_w$ depends on $w$ as follows:
\[
C_w = C\nor{w}_{L^{\infty}\LRp{D^k}}\nor{\frac{1}{w}}_{L^{\infty}\LRp{D^k}} \nor{w}_{W^{N+1,\infty}\LRp{D^k}}.
\]
\label{lemma:mult}
\end{theorem}
\begin{proof}
The first bound is a direct application of Theorem~\ref{thm:wproj} to
\[
\nor{{u}{w}-T_{w}u}_{L^2\LRp{D^k}} \leq \nor{w}_{L^{\infty}\LRp{D^k}} \nor{{u}-\frac{1}{w}\Pi_N\LRp{{u}{w}}}_{L^2\LRp{D^k}}.
\]
This second bound is derived by bounding first the projection error of $uw$ and the deviation of $T^{-1}_{1/w} u$ from $\Pi_N\LRp{{u}{w}}$.  The introduction of $\Pi_N \LRp{{u}{w}}$ allows us to use the fact that $T_{1/w}^{-1} T_{1/w} = I$ over $P^N$.
\[
\nor{{u}{w}-T^{-1}_{1/w} u}_{L^2\LRp{D^k}} \leq \nor{{u}{w}-\Pi_N\LRp{{u}{w}}}_{L^2\LRp{D^k}} + \nor{\Pi_N\LRp{{u}{w}}-T^{-1}_{1/w} u}_{L^2\LRp{D^k}}
\]
The former term is bounded by the standard interpolation estimate and regularity of $u$ and $w$.  The latter term can be bounded as follows:
\begin{align*}
&\nor{T^{-1}_{1/w} u-\Pi_N\LRp{{u}{w}}}_{L^2\LRp{D^k}} = \nor{T^{-1}_{1/w} {\Pi_N\LRp{{u}}} - T_{1/w}^{-1}T_{1/w}\Pi_N \LRp{{u}{w}}}_{L^2\LRp{D^k}}\\
&\leq \nor{T_{1/w}^{-1}} \nor{\Pi_N \LRp{{u}} - \Pi_N \LRp{\frac{1}{w}\Pi_N\LRp{{u}{w}}}}_{L^2\LRp{D^k}}\\
&\leq \nor{w}_{L^{\infty}\LRp{D^k}} \nor{\Pi_N}_{L^2\LRp{D^k}} \nor{u - {\frac{1}{w}\Pi_N\LRp{{u}{w}}}}_{L^2\LRp{D^k}} \\
&\leq {C} h^{N+1}\nor{w}_{L^{\infty}\LRp{D^k}}\nor{\frac{1}{w}}_{L^{\infty}\LRp{D^k}} \nor{w}_{W^{N+1,\infty}\LRp{D^k}}  \nor{u}_{W^{N+1,2}\LRp{D^k}},
\end{align*}
where we have used Lemma~\ref{lemma:properT} and the fact that $\nor{\Pi_N}_{L^2} = 1$ for affinely mapped elements.  
\end{proof}

Finally, we give an estimate for moments of the difference between the weighted and weight-adjusted inner products:
\begin{theorem}
Let $u\in W^{N+1,2}\LRp{D^k}$, $w\in W^{N+1,\infty}\LRp{D^k}$, and $v \in P^M\LRp{D^k}$ for $0 \leq M\leq N$; then
\begin{align*}
&\LRb{\LRp{{w}u,v}_{L^2\LRp{D^k}} - \waip{u,v}{1/w}} \\
&\leq Ch^{2N+2 - M} \nor{{w}}_{L^{\infty}\LRp{D^k}}\nor{\frac{1}{w}}_{L^{\infty}\LRp{D^k}}^2 \nor{w}^2_{W^{N+1,\infty}\LRp{D^k}}  \nor{u}_{W^{N+1,2}\LRp{D^k}} \nor{v}_{L^{\infty}\LRp{D^k}}.
\end{align*}
\label{lemma:cons}
\end{theorem}
\begin{proof}
Over each element $D^k$, the weight-adjusted inner product gives
\[
\waip{u,v}{1/w} = \LRp{T^{-1}_{1/w} u,v}_{L^2\LRp{D^k}} = \LRp{\frac{1}{w}wT^{-1}_{1/w} u,v}_{L^2\LRp{D^k}}.
\]
If ${w}$ is polynomial of degree $N-M$, then $\LRp{\frac{1}{w}T^{-1}_{1/w} u,{v}{w}}_{L^2} = \LRp{u,{v}{w}}_{L^2}$ and the moment of the difference is zero.  If ${w} \not\in P^{N-M}\LRp{D^k}$, then $\LRp{\frac{1}{w}T^{-1}_{1/w} u,{v}{w}}_{L^2} \neq \LRp{u,{v}{w}}_{L^2}$.  To bound the difference, we can add and subtract the projection of ${vw}$ onto $P^N$ 
\begin{align*}
 &\LRp{\frac{1}{w}wT^{-1}_{1/w} u,v}_{L^2\LRp{D^k}} \\
 &= \LRp{\frac{1}{w}T^{-1}_{1/w} u, {v}{w}-\Pi_N\LRp{{v}{w}}}_{L^2\LRp{D^k}} + \LRp{ \frac{1}{w}T^{-1}_{1/w} u, \Pi_N\LRp{{v}{w}}}_{L^2\LRp{D^k}} \\
 &= \LRp{\frac{1}{w}T^{-1}_{1/w} u, {v}{w}-\Pi_N\LRp{{v}{w}}}_{L^2\LRp{D^k}} + \LRp{ u,\Pi_N\LRp{{v}{w}}}_{L^2\LRp{D^k}}.
\end{align*}
The difference then becomes
\begin{align*}
&\LRb{ \LRp{u,{v}{w}}_{L^2\LRp{D^k}} - \waip{u,v}{1/w}} \\
&= \LRb{\LRp{u,{v}{w} - \Pi_N\LRp{{v}{w}}}_{L^2\LRp{D^k}} + \LRp{\frac{1}{w}T_{1/w}^{-1} u,{v}{w}-\Pi_N\LRp{{v}{w}}}_{L^2\LRp{D^k}}}\\ 
&\leq \LRb{\LRp{u - \frac{1}{w}T^{-1}_{1/w} u,{v}{w}-\Pi_N\LRp{{v}{w}}}_{L^2\LRp{D^k}}} \\
&\leq \nor{u-\frac{1}{w}T^{-1}_{1/w} u}_{L^2\LRp{D^k}} \nor{{v}{w}-\Pi_N\LRp{{v}{w}}}_{L^2\LRp{D^k}}.
\end{align*}
For $vw$ sufficiently regular, the Bramble-Hilbert lemma implies
\[
\nor{{v}{w}-\Pi_N\LRp{{v}{w}}}_{L^2\LRp{D^k}} \leq C\sqrt{J^k} \LRb{{v}{w}}_{W^{N+1,2}\LRp{\widehat{D}}}.
\]
By quasi-regularity of $D^k$ and the Leibniz product rule, the seminorm can be bounded by
\[
\LRb{{v}{w}}_{W^{N+1,2}\LRp{\widehat{D}}} \leq C\frac{1}{\sqrt{J^k}}h^{N+1}\nor{v}_{W^{N+1,2}\LRp{D^k}}\nor{w}_{W^{N+1,\infty}\LRp{D^k}}.
\]
Applying a scaling argument for $v\in P^{M}\LRp{D^k}$ and Bernstein's inequality \cite{sarantopoulos1991bounds} then yields
\[
\nor{v}_{W^{N+1,2}\LRp{D^k}} \leq C_B h^{-M} \nor{v}_{L^{\infty}\LRp{D^k}}.
\]
where $C_B$ is a constant depending on $N$.  This implies that 
\[
\nor{{v}{w}-\Pi_N\LRp{{v}{w}}}_{L^2\LRp{D^k}} \leq C h^{N+1-M}\nor{w}_{W^{N+1,2}\LRp{D^k}}\nor{v}_{L^{\infty}\LRp{D^k}}.
\]
We can then use Theorem~\ref{lemma:mult} to bound the remaining term
\begin{align*}
&\nor{u-\frac{1}{w}T^{-1}_{1/w}u}_{L^2\LRp{D^k}}\\
 &= \nor{\frac{1}{w}\LRp{{u}{w}-T^{-1}_{1/w}u}}_{L^2\LRp{D^k}}  \leq \nor{\frac{1}{w}}_{L^{\infty}\LRp{D^k}}\nor{{u}{w} - T^{-1}_{1/w} u}_{L^2\LRp{D^k}}\\
&\leq C h^{N+1}\nor{{w}}_{L^{\infty}\LRp{D^k}}\nor{\frac{1}{w}}^2_{L^{\infty}\LRp{D^k}} \nor{w}_{W^{N+1,\infty}\LRp{D^k}} \nor{u}_{W^{N+1,2}\LRp{D^k}}.
\end{align*}
Combining these two estimates 
gives the desired bound.
\end{proof}

\section{A low storage weight-adjusted DG method}
\label{sec:wadg}
Using the weight-adjusted inner product, we can now introduce the weight-adjusted DG method.  Recall the DG formulation of the pressure equation introduced in Section~\ref{sec:form} 
\begin{align*}
\int_{D^k} \frac{1}{c^2}\pd{p}{t}{}v \diff x &= -\int_{D^k} \Div\bm{u}v \diff x + \int_{\partial D^k} \frac{1}{2}\LRp{\tau_p\jump{p} - \bm{n}\cdot \jump{\bm{u}} }v^- \diff x, \quad \forall v\in P^N\LRp{D^k}.
\end{align*}
The standard DG method is shown to be energy stable with respect to the $L^2$ norm weighted by $1/c^2$ which appears on the left hand of the pressure equation, which corresponds to the weighted $L^2$ inner product 
\[
\wip{p,v}{w} =  \int_{D^k} T_{w} p v = \int_{D^k} w p v = \qquad \forall v\in P^N\LRp{D^k}.
\]
where 
\[
w(\bm{x}) = 1/c^2(\bm{x}).
\]
For the remainder of this paper, we will assume this specific definition of $w(x)$ for the acoustic wave equation.  Motivated by the fact that $T_{1/w}^{-1}u \approx wu$, the weight-adjusted DG method approximates the weighted left hand side inner product in the DG pressure equation with the weight-adjusted inner product in Section~\ref{sec:ip} 
\begin{align*}
\int_{D^k} T^{-1}_{1/w}\LRp{\pd{p}{t}} v\diff x &= -\int_{D^k} \LRp{\Div\bm{u}}v \diff x + \int_{\partial D^k} \frac{1}{2}\LRp{\tau_p\jump{p} - \bm{n}\cdot \jump{\bm{u}} }v^- \diff x.
\end{align*}
We note that the constants appearing in the bounds for Theorem~\ref{lemma:mult} are identical for both $T_{w}$ and $T^{-1}_{1/w}$, which suggests that the behavior of the weight-adjusted DG method should be very similar to that of the standard DG method.  

A crucial aspect of the weight-adjusted DG method is that it is energy stable, due to the use of an equivalent inner product in the DG pressure equation.
Repeating the analysis in Section~\ref{sec:energy} for the weight-adjusted DG method gives that
\begin{align}
\pd{}{t}\sum_{k}\int_{D^k} \LRp{T_{1/w}^{-1}p}p    + \LRb{\bm{u}}^2 = - \sum_{k}\frac{1}{2}\int_{\partial D^k}  \tau_p \jump{p}^2 + \tau_u \jump{\bm{u}}^2 < 0,
\label{eq:stability2}
\end{align}
and since
\[
\sum_{k}\int_{D^k} \LRp{T_{1/w}^{-1}p} p = \sum_{k}\LRp{T_{1/w}^{-1}p,p}_{L^2\LRp{D^k}} = \sum_k \waip{p,p}{1/w} > 0
\]
for $w = 1/c^2$.  The left hand side of (\ref{eq:stability2}) implies that a squared norm on $\LRp{p,\bm{u}}$ is non-increasing in time.  Additionally, by Lemma~\ref{lemma:equiv}, this normed quantity is equivalent to the $L^2$ norm of $p,\bm{u}$ with the same equivalence constants as the weighted $L^2$ inner product used in (\ref{eq:stability}) for the standard DG method.  

By replacing the weighted inner product on the left hand side with an approximation, a different mass matrix $\tilde{\bm{M}}^k$ is induced
\[
\LRp{\tilde{\bm{M}}^k}_{ij} = \waip{\phi_j,\phi_i}{1/w}. 
\]
For polynomial functions $u$ on an element $D^k$ with expansion coefficients $\bm{u}$, computing $u_w = T_{1/w}^{-1}u$ reduces to a square matrix multiplication
\[
\bm{u}_w = \LRp{\bm{M}^k_{1/w}}^{-1} \bm{M}^k \bm{u},
\]
where  $\bm{u}_w$ are the expansion coefficients of $u_w$ and $\bm{M}_{1/w}^k$ is defined entrywise
\[
\LRp{\bm{M}_{1/w}^k}_{ij} = \int_{D^k} \frac{1}{w}\phi_j \phi_i.
\]
Thus, the Gram matrix associated with the weight-adjusted inner product has the form
\[
\tilde{\bm{M}}^k = \bm{M}^k\LRp{\bm{M}^k_{1/w}}^{-1}{\bm{M}^k},
\]
resulting in a discrete formulation for the weight-adjusted DG method 
\begin{align*}
\bm{M}^k\LRp{\bm{M}^k_{1/w}}^{-1}{\bm{M}^k}\td{\bm{p}}{t} &=\sum_{i = 1,2,3}\bm{S}_{i}^k \bm{U}_j + \sum_{f=1}^{N_{\text{faces}}}\bm{M}^k_f F_p(\bm{p}^-,\bm{p}^+,\bm{U}^-,\bm{U}^+),\\
\bm{M}^k\td{\bm{U}_{x_i}}{t} &= \bm{S}_{i}^k \bm{p} + \sum_{f=1}^{N_{\text{faces}}}  \bm{n}_{i}\bm{M}^k_f  F_{u}(\bm{p}^-,\bm{p}^+,\bm{U}^-,\bm{U}^+), \qquad i = 1,2,3.
\end{align*}
For hexahedral elements with quadrature-based mass-lumping, this procedure reduces to collocation of $w(x) = 1/c^2(x)$ at quadrature points.  For tetrahedral elements (which do not admit high order mass lumped schemes under a $P^N$ approximation space \cite{chin1999higher,cohen2001higher}), this method provides a low storage implementation through the fact that
\[
\LRp{\bm{M}^k\LRp{\bm{M}^k_{1/w}}^{-1}{\bm{M}^k} }^{-1} = \LRp{\bm{M}^k}^{-1}{\bm{M}^k_{1/w}}\LRp{\bm{M}^k}^{-1}.
\]
For planar tetrahedra (and other affinely mapped elements), $\LRp{\bm{M}^k}^{-1} = \frac{1}{J^k} \widehat{\bm{M}}^{-1}$, requiring storage of only the reference array $\widehat{\bm{M}}^{-1}$.  The application of $\bm{M}_{1/w}^k$ can be done in a matrix-free manner: for $u \in P^N$ with expansion coefficients $\bm{u}$, 
\[
\LRp{\bm{M}_{1/w}^k \bm{u}}_i = \int_{\widehat{D}} \frac{1}{w} u \phi_i J^k.
\]
Each integral can be computed over the reference element using quadrature, requiring only $O(N^3)$ storage for values of ${c^2}$ at nodal or quadrature points.  

We introduce the weak differentiation matrices and lift matrices $\bm{L}^k_f$ for the face $f$ of $D^k$ 
\begin{align*}
\bm{D}_x =  \LRp{\bm{M}^k}^{-1}\bm{S}_x, \qquad
\bm{D}_y = \LRp{\bm{M}^k}^{-1}\bm{S}_y, \qquad
\bm{D}_z = \LRp{\bm{M}^k}^{-1}\bm{S}_z, \qquad
\bm{L}^k_f = \LRp{\bm{M}^k}^{-1}\bm{M}^k_f.
\end{align*}
For planar tetrahedra, these differentiation and lift matrices can be computed from linear combinations and scalings of reference derivative and lift matrices \cite{hesthaven2007nodal}.  The matrix form of the semi-discrete weight-adjusted DG pressure equation can then be written as
\begin{align}
\td{\bm{p}}{t} =\LRp{\bm{M}^k}^{-1}{\bm{M}^k_{1/w}} \LRp{\sum_{i = 1,2,3}\bm{D}_{i}^k \bm{U}_j + \sum_{f=1}^{N_{\text{faces}}}\bm{L}^k_f F_p(\bm{p}^-,\bm{p}^+,\bm{U}^-,\bm{U}^+)},
\label{eq:WadgDiscretePressure}
\end{align}
where we have referred to the weak differentiation matrices $\LRc{\bm{D}_x,\bm{D}_y,\bm{D}_z}$ as $\LRc{\bm{D}_1,\bm{D}_2,\bm{D}_3}$ for succinctness.  We note that (for an appropriate choices of flux $F_p$) the contribution
\begin{align}
\LRp{\sum_{i = 1,2,3}\bm{D}_{x_i}^k \bm{U}_j + \sum_{f=1}^{N_{\text{faces}}}\bm{L}^k_f F_p(\bm{p}^-,\bm{p}^+,\bm{U}^-,\bm{U}^+)}
\label{eq:wadgDiv}
\end{align}
is simply the the DG discretization of the divergence operator and the pressure equation DG right hand contribution for wave propagation in homogeneous media.  This illustrates the fact that implementation of the weight-adjusted DG method is relatively non-invasive.  For example, a time-domain DG code with explicit timestepping for homogeneous media typically involves one step to compute right hand side contributions and one step to evolve the solution in time using a time integration scheme.  For such a code, the weight-adjusted DG method for heterogeneous media could be implemented by adding only a single additional step which applies $\LRp{\bm{M}^k}^{-1}{\bm{M}^k_{1/w}}$ to the right hand side (for homogeneous media) before time integration.  

\subsection{Consistent scaling by $c^2$}

The strong form of the pressure equation can be rescaled by $c^2$ 
\begin{align}
\pd{p}{t} + c^2 \Div u = 0.
\label{eq:repressure}
\end{align}
However, since this is in non-conservative form, it is non-trivial to derive appropriate formulations and numerical fluxes which result in an energy stable DG methods.  

As suggested by (\ref{eq:WadgDiscretePressure}) and (\ref{eq:wadgDiv}), the weight-adjusted DG method can be interpreted as a way to consistently rescale by $c^2$ while maintaining the conservative form of the wave equation.  Recall the definition of the lift operator $L^k_f$ for a given face $f$ of the element $D^k$
\[
\LRp{ L^k_f(u), v}_{ D^k} = \LRp{ u, v}_{\partial D^k_f}, \qquad v \in V_h\LRp{D^k}.
\]
The weight-adjusted DG formulation can then be expressed using $L_f^k$ as 
\begin{align}
\int_{D^k} T_{1/w}^{-1}{\pd{p}{t}} v\diff x &+ \int_{D^k} \LRp{\Div\bm{u} -\sum_{f}    L^k_f\LRp{F_p({p}^-,{p}^+,\bm{u}^-,\bm{u}^+)}}v \diff x = 0  \nonumber \\
\int_{D^k} \pd{\bm{u}}{t}{}\bm{\tau} \diff x &+ \int_{D^k} \LRp{\Grad p - \sum_f  L^k_f\LRp{F_{u}({p}^-,{p}^+,\bm{u}^-,\bm{u}^+)} \bm{n}^- }\bm{\tau} \diff x  = 0.
\end{align}
This is sometimes written in a more compact form 
\begin{align}
\LRp{T_{1/w}^{-1}\pd{p}{t},v }_{L^2\LRp{D^k}} &+ \LRp{ \Grad_h \cdot \bm{u},v}_{L^2\LRp{D^k}}=0\\
\LRp{\pd{\bm{u}}{t},\bm{\tau}}_{L^2\LRp{D^k}} &+\LRp{ \Grad_h p,\bm{\tau}}_{L^2\LRp{D^k}}=0, \qquad (v,\bm{\tau}) \in V_h \times \LRp{V_h}^d.
\end{align}
where we have introduced the weak DG gradient and divergence  $\Grad_h, \Grad_h\cdot$.  These weak DG differential operators are defined such that their restriction to an element $D^k$ yields
\begin{align}
\left.\Grad_h \cdot p\right|_{D^k} &= \left.\LRp{\Div p}\right|_{D^k} - \sum_{f} L^k_f\LRp{F_p(p,\bm{u})}\nonumber\\
\left.\Grad_h \bm{u}\right|_{D^k} &= \left.\LRp{\Grad \bm{u}}\right|_{D^k} - \sum_{f} \bm{n}^-L^k_f\LRp{F_u(p,\bm{u})},  
\label{eq:consistentrescale}
\end{align}
where $F_p(p,\bm{u}), F_u(p,\bm{u})$ are the numerical fluxes over a face $f$.  
The weight-adjusted DG method can be derived using the weak DG divergence in (\ref{eq:consistentrescale}) instead of the exact divergence.  Replacing the strong divergence of (\ref{eq:repressure}) with the weak DG divergence, then multiplying both sides of by a test function in $V_h$ and integrating results in the weight-adjusted DG formulation.  This incorporates the scaling by $c^2$ in a consistent manner, multiplying terms within volume integrals only.  Without introducing the lift operator, it is not immediately clear how to incorporate the scaling by $c^2$ within surface integrals.

\subsection{Convergence}

With the estimates in Section~\ref{sec:estimates} and consistency of the formulation, it is possible to derive \textit{a priori} error estimates for the weight-adjusted DG method.  We follow the approach of \cite{warburton2013low} to obtain an $O\LRp{h^{N+1/2}}$ bound on the $L^2$ error.  

For functions $u \in L^2\LRp{\Omega}$ such that $\left.u\right|_{D^k} \in W^{N+1,2}\LRp{D^k}$, we define the norm
\[
\nor{u}_{W^{N+1,p}\LRp{\Omega_h}} = \LRp{ \sum_k \nor{u}_{W^{N+1,p}\LRp{D^k}}^2}^{1/2}.
\]
We consider solutions $\LRp{p,\bm{u}} \in W^{N+1,2}\LRp{\Omega_h} \times \LRp{W^{N+1,2}\LRp{\Omega_h}}^d$ over the time interval $[0,T]$ such that
\begin{align*}
\sup_{t' \in [0,T]}\nor{p}_{W^{N+1,2}\LRp{\Omega_h}} &< \infty, \qquad \sup_{t' \in [0,T]}\nor{\bm{u}}_{W^{N+1,2}\LRp{\Omega_h}} < \infty,\\
\sup_{t' \in [0,T]}\nor{\pd{p}{t}}_{W^{N+1,2}\LRp{\Omega_h}} &< \infty, \qquad \sup_{t' \in [0,T]}\nor{\pd{\bm{u}}{t}}_{W^{N+1,2}\LRp{\Omega_h}} < \infty.
\end{align*}
Under these regularity assumptions,\footnote{These assumptions may be relaxed somewhat using techniques from \cite{grote2007interior}.} the following theorem holds for $p$ and the components $\bm{u}_i$ of the velocity:
\begin{theorem}[Theorem 3.3 of \cite{warburton2013low}]
\begin{align*}
\nor{p - \Pi_N p}_{\partial D^k} &\leq C h^{N+1/2}\nor{p}_{W^{N+1,2}(D^k)}\\
\nor{\bm{u}\cdot \bm{n} - \Pi_N \bm{u}\cdot \bm{n}}_{\partial D^k} &\leq C h^{N+1/2}\nor{\bm{u}}_{W^{N+1,2}(D^k)}, \qquad i = 1,2,3.
\end{align*}
\label{thm:tracereg}
\end{theorem}
\note{
We will also use the following modified Gronwall's inequality
\begin{lemma}[Lemma 1.10 in \cite{dolejvsi2015discontinuous}]
Suppose that $a,b,c,d \in C[0,T]$ are non-negative functions and that
\[
a^2(t) + b(t) \leq c(t) + 2\int_0^t d(s) a(s) \diff{s}, \qquad \forall t\in [0,T].
\]
Then, for any $t\in [0,T]$,
\[
\sqrt{a^2(t) + b(t)} \leq \sup_{s\in[0,t]} \sqrt{c(s)} + \int_0^t d(s) \diff{s}.
\]
\label{lemma:gron}
\end{lemma}

Then, we have the following \textit{a priori} estimate for the weight-adjusted DG solution $\LRp{p_h,\bm{u}_h}$ at time $T$
\begin{theorem}
\begin{align*}
&\nor{\LRp{p(\bm{x},T),\bm{u}(\bm{x},T)} -\LRp{p_h(\bm{x},T),\bm{u}_h(\bm{x},T)}}_{\L} \leq \\
&\LRp{C_1 + C_2 T} h^{N+1/2}\sup_{t'\in [0,T]}\LRp{ \nor{\LRp{p,\bm{u}}}_{W^{N+1,2}(\Oh)} + h^{1/2}\nor{\frac{1}{c^2}}_{W^{N+1,\infty}\LRp{\Oh}}\nor{\pd{}{t}\LRp{p,\bm{u}}}_{W^{N+1,2}(\Oh)}},
\end{align*}
where $C_2$ depends on $c_{\min},c_{\max}$.  
\end{theorem}
}
\begin{proof}
We introduce group variables $U = \LRp{p,\bm{u}}$ and $V = \LRp{v,\bm{\tau}}$ to rewrite the variational formulation as
\[
\LRp{\pd{U}{t},V}_{w,\Omega} + a(U,V) + b(U,V) = 0,
\]
where $\LRp{U,V}_{w,\Omega}$ is 
\[
\LRp{U,V}_{w,\Omega} = \sum_{k} \waip{p,v}{c^2} + \LRp{\bm{u},\bm{\tau}}_{L^2\LRp{D^k}}.
\]
The volume and surface contributions to the formulation are given by
\begin{align*}
a(U,V) &= \sum_k \int_{D^k} -\bm{u}\cdot\Grad v +  \Grad p\cdot \bm{\tau}\\
b(U,V) &= \sum_{k}\int_{\partial D^k} \LRp{\frac{\tau_p}{2}\jump{p} - \avg{\bm{u}}\cdot\bm{n}^-} v + \frac{1}{2}\LRp{{\tau_u}\jump{u}\cdot\bm{n}^- - \jump{p}} \bm{\tau}\cdot \bm{n}^-.
\end{align*}
The proof of energy stability implies that $b(U,V)$ is positive semi-definite, and that
\begin{align*}
b(U,U) &= \frac{1}{2}\sum_k \tau_p \nor{\jump{p}}^2_{L^2\LRp{\partial D^k}} + \tau_u \nor{\jump{\bm{u}}\cdot \bm{n}}^2_{L^2\LRp{\partial D^k}}\\
\frac{1}{2}\pd{}{t}(U,U)_{w,\Omega} &= -b(U,U).
\end{align*}
Let $\Pi_h: L^2\LRp{\Omega} \rightarrow \bigoplus_{k} P^N\LRp{D^k}$ be the $L^2$ projection onto the space of degree $N$ polynomials over the triangulation $\Omega_h$.  The error $E$ between the exact solution $U$ and the the weight-adjusted DG solution $U_h = \LRp{p_h,\bm{u}_h}$ can be defined in terms of the interpolation error $\epsilon$ and discretization error $\eta$
\[
E = U-U_h = \LRp{U-\Pi_h U} + \LRp{\Pi_h U - U_h} = \epsilon + \eta.
\]
Since the interpolation error $\epsilon$ can be bounded by regularity assumptions, what remains is to bound the discretization error $\eta = \Pi_h\LRp{U-U_h}$ at time $T$.  

\note{
Assuming sufficient regularity \cite{brenner2007mathematical, hesthaven2007nodal}, the exact solution at time $T$ satisfies the DG formulation (\ref{eq:form}) with weighted $L^2$ inner product
\begin{align*}
\LRp{\frac{1}{c^2}\pd{p}{t},v}_{\Omega} + \LRp{\pd{\bm{u}}{t},\bm{\tau}}_{\Omega} + a(U,V) + b(U,V) &= 0, \qquad \forall V\in V_h,
\end{align*}
while the discrete solution satisfies the WADG formulation
\begin{align*}
\LRp{\pd{U_h}{t},V}_{w,\Omega} + a(U_h,V) + b(U_h,V) &= 0, \qquad \forall V\in V_h.
\end{align*}
Taking $V = \eta$, subtracting these two equations and rearranging yields the error equation
\begin{equation}
\LRp{\frac{1}{c^2}\pd{p}{t},\eta_p}_{\Omega} + \LRp{\pd{\bm{u}}{t},\bm{\eta}_u}_{\Omega} - \LRp{\pd{U_h}{t},\eta}_{w,\Omega} + b(\eta,\eta) = a(\epsilon,\eta) + b(\epsilon,\eta).
\label{eq:erroreq}
\end{equation}
where we have used $a(\eta,\eta) = 0$ by skew-symmetry.  Integrating by parts gives
\[
a(\epsilon,\eta) = \sum_k \int_{D^k} -\bm{\epsilon}_u\Grad \eta_p -  \epsilon_p \Div \bm{\eta}_u + \int_{\partial D^k} p^- \bm{\eta}_u\cdot \bm{n},
\]
where $\epsilon_p, \bm{\epsilon}_u$ and $\eta_p,\bm{\eta}_u$ are the $p$ and $\bm{u}$ components of the interpolation and discretization error, respectively.  For affinely mapped elements, $\Div\bm{\eta}_u, \Grad\eta_p$ are polynomial, and volume terms disappear through orthogonality of the $L^2$ projection to polynomials up to degree $N$.  We can then bound the contribution by combining contributions over shared faces and applying the arithmetic-geometric mean inequality
\begin{align*}
a(\epsilon,\eta) + b(\epsilon,\eta) &= \frac{1}{2} \sum_{k}\int_{\partial D^k} \LRp{\frac{\tau_p}{2}\jump{\epsilon_p} - \avg{\bm{\epsilon}_u}\cdot\bm{n}^-}\jump{ \eta_p} + \LRp{\frac{\tau_u}{2}\jump{\bm{\epsilon}_u}\cdot\bm{n}^- - \avg{\epsilon_p}} \jump{\bm{\eta}_u}\cdot \bm{n}^- \\
&\leq \tilde{C_\tau} \sum_{k}\int_{\partial D^k} \LRp{\jump{\epsilon_p} - \avg{\bm{\epsilon}_u}\cdot\bm{n}^-} \frac{\tau_p}{2}\jump{ \eta_p} + \LRp{\jump{\bm{\epsilon}_u}\cdot\bm{n}^- - \avg{\epsilon_p}} \frac{\tau_u}{2}\jump{\bm{\eta}_u}\cdot \bm{n}^- \\
&\leq C_\tau \sum_k\int_{\partial D^k} \LRb{\epsilon}^2 \LRp{ \frac{\tau_p}{2}\nor{\jump{ \eta_p}}^2_{L^2\LRp{\partial D^k}} + \frac{\tau_u}{2}\nor{\jump{\bm{\eta}_u}\cdot \bm{n}}^2_{L^2\LRp{\partial D^k}} }.  
\end{align*}
Applying Young's inequality with $\alpha = C_\tau/2$ then gives
\[
\LRb{a(\epsilon,\eta) + b(\epsilon,\eta)} \leq b(\eta,\eta) + \frac{C_\tau^2}{4} \sum_k \nor{{\epsilon}}^2_{L^2\LRp{\partial D^k}}.
\]
Terms involving the time derivative of pressure can be controlled by introducing the $L^2$ projection and using properties of $T_{c^2}^{-1}$
\begin{align*}
\LRp{\frac{1}{c^2}\pd{p}{t}-T_{c^2}^{-1} \Pi_h \pd{p}{t},\eta_p}_{\Omega} &= \LRp{\frac{1}{c^2}\pd{p}{t} - T_{c^2}^{-1} \Pi_h \pd{p}{t},\eta_p}_{\Omega} +  \LRp{T_{c^2}^{-1} \Pi_h\pd{p}{t}-T_{c^2}^{-1} \pd{p_h}{t},\eta_p}_{\Omega} \\
&= \LRp{\pd{\delta_p}{t},\eta_p}_{\Omega} +  \LRp{T_{c^2}^{-1}\pd{\eta_p}{t} ,\eta_p}_{\Omega} = \LRp{\pd{\delta_p}{t},\eta_p}_{\Omega} +  \frac{1}{2}\pd{}{t}\LRp{T_{c^2}^{-1}\eta_p ,\eta_p}_{\Omega},
\end{align*}
where $\delta_p = \frac{1}{c^2}p - T_{c^2}^{-1} \Pi_h p= \frac{1}{c^2}p - T_{c^2}^{-1} p$ is the WADG consistency error in the pressure variable.  Terms involving time derivatives of velocity satisfy
\[
\LRp{\pd{\bm{u}}{t},\bm{\eta}_u}_{\Omega} - \LRp{\pd{\bm{u}_h}{t},\bm{\eta}_u}_{\Omega} = \LRp{\pd{\bm{\eta}_u}{t},\bm{\eta}_u}_{\Omega} + \LRp{\pd{\bm{\epsilon}_u}{t},\bm{\eta}_u}_{\Omega}.
\]
Combining these gives 
\begin{align*}
\LRp{\frac{1}{c^2}\pd{p}{t},\eta_p}_{\Omega} + \LRp{\pd{\bm{u}}{t},\bm{\eta}_u}_{\Omega} - \LRp{\pd{U_h}{t},\eta}_{w,\Omega} &= \pd{}{t}\frac{1}{2}\LRp{\eta,\eta}_{\Omega} + \LRp{\pd{\delta}{t},\eta}_{\Omega}.  
\end{align*}
where 
\[
\LRp{\pd{\delta}{t},\eta}_{\Omega} = \LRp{\pd{\delta_p}{t},\eta_p}_{\Omega} + \LRp{\pd{\bm{\epsilon}_u}{t},\bm{\eta}_u}_\Omega.  
\]
Substituting these expressions into the error equation (\ref{eq:erroreq}) gives 
\begin{align*}
\pd{}{t}\frac{1}{2} (T_{c^2}^{-1}\eta,\eta)_\Omega + b(\eta,\eta) &\leq \LRb{\LRp{\pd{\delta}{t},\eta}_{\Omega}} + b(\eta,\eta) + \frac{C_\tau^2}{4} \sum_{k} \nor{\epsilon}_{L^2\LRp{\partial D^k}}^2.
\end{align*}
We eliminate factors of $\frac{1}{2}$ and $b(\eta,\eta)$ on both sides.  Then, integrating over $[0,T]$, applying Theorem~\ref{lemma:equiv}, and using Cauchy-Schwarz yields
\[
\frac{1}{c_{\max}} \nor{\eta}_{L^2\LRp{\Omega}}^2 \leq \int_0^T \nor{\eta}_{L^2\LRp{\Omega}}\nor{\pd{\delta}{t}}_{L^2\LRp{\Omega}}  + \frac{C_\tau^2}{2} \sum_{k} \nor{\epsilon}_{L^2\LRp{\partial D^k}}^2.
\]
The modified Gronwall inequality then yields a bound on $\nor{\eta}_{L^2\LRp{\Omega}}$
\begin{align*}
\nor{\eta}_{L^2\LRp{\Omega}} &\leq \tilde{C}\int_0^T \nor{\pd{\delta}{t}}_{L^2\LRp{\Omega}}  + \sup_{t'\in [0,T]} \sqrt{\int_0^T \frac{C_\tau^2}{2} \sum_{k} \nor{\epsilon}_{L^2\LRp{\partial D^k}}^2}\\
&\leq  CT \sup_{t'\in [0,T]} \LRp{ \nor{\pd{\delta}{t}}_{L^2\LRp{\Omega}}  + \sqrt{\sum_{k} \nor{\epsilon}_{L^2\LRp{\partial D^k}}^2}}.
\end{align*}
where $C$ depends on $ c_{\max}$ and the penalty parameters.  
The right hand side terms are then bounded using regularity assumptions.  The time derivative term is bounded using Theorem~\ref{lemma:mult}
\begin{align*}
\nor{\pd{\delta}{t}}_{L^2\LRp{\Omega}} &\leq \nor{\pd{\delta_p}{t}}_{L^2\LRp{\Omega}} + \nor{\pd{\bm{\epsilon}_u}{t}}_{L^2\LRp{\Omega}}\\
&\leq \nor{\pd{\delta_p}{t}}_{L^2\LRp{\Omega}} + Ch^{N+1} \nor{\pd{\bm{u}}{t}}_{W^{N+1,2}\LRp{\Oh}}\\
&\leq C\frac{c_{\max}}{c_{\min}} \nor{\frac{1}{c^2}}_{W^{N+1,\infty}\LRp{\Oh}} h^{N+1} \nor{\pd{}{t}(p,\bm{u})}_{W^{N+1,2}\LRp{\Oh}}.
\end{align*}
while the trace term is bounded using Theorem~\ref{thm:tracereg}
\[
\sqrt{\sum_k \nor{\epsilon}_{\partial D^k}^2} \leq \sqrt{C \sum_k h^{2N+1}\nor{\LRp{p,\bm{u}}}^2_{W^{N+1,2}(D^k)}} \leq Ch^{N+1/2} \nor{\LRp{p,\bm{u}}}_{W^{N+1,2}(\Oh)}
\]
Taking the supremum over $[0,T]$ and applying the triangle inequality to $U-U_h = \epsilon + \eta$ completes the proof.  
}
\end{proof}

\subsection{Local conservation}
\label{sec:conservation}

While standard DG methods are locally conservative, the use of the weight-adjusted mass matrix does not preserve local conservation of the same quantities conserved by the standard DG method.  However, Theorem~\ref{lemma:cons} gives an estimate which implies a higher order $O(h^{2N+2})$ convergence of the conservation error for smooth solutions.  Since conservation conditions for DG are recovered by testing with piecewise constant test functions \cite{ellis2014locally}, we define the local conservation error as the $M=0$ moment of the error between the standard DG and weight-adjusted DG inner products for polynomial $u$, summed over all elements $D^k$
\begin{align*}
&\sum_k \LRb{\LRp{\frac{1}{c^2}u,1}_{L^2\LRp{D^k}} - \LRp{u,1}_{T^{-1}_{c^2}}} \\
&\leq Ch^{2N+2}  \nor{c^2}^2_{L^{\infty}\LRp{\Oh}}\sup_k\nor{\frac{1}{c^2}}^2_{W^{N+1,\infty}\LRp{D^k}}{\sum_k \nor{u}_{W^{N+1,2}\LRp{\Oh}}}  .
\end{align*}
We note that the above bound depends on the regularity of both $c^2$ and the solution $u$.  As noted in the proof of Theorem~\ref{lemma:cons}, it is possible to restore local conservation by replacing $c^2$ with its degree $N$ polynomial projection or interpolant on each element, though this can introduce an error if $c^2$ is poorly approximated by $P^N\LRp{D^k}$.

Alternatively, it is also simple to restore conservation through a rank-one update to the mass matrix.  Let $\bm{e}$ be the vector of degrees of freedom representing a constant; then, we seek $\alpha \bm{vv}^T$ such that
\[
{\LRp{\bm{M}^k\LRp{\bm{M}^k_{c^2}}^{-1}\bm{M}^k + \alpha \bm{v}\bm{v}^T}\bm{e} - \bm{M}^k_{1/c^2}\bm{e}} = 0.
\]
This implies that $\bm{v}$ is the conservation error up to a scaling constant.  This constant can be determined as follows: define
\[
\bm{v} = \LRp{\bm{M}^k\LRp{\bm{M}^k_{c^2}}^{-1}\bm{M}^k - \bm{M}^k_{1/c^2}}\bm{e}
\]
Multiplying by $\bm{e}^T$ on the left gives 
\[
\bm{e}^T \bm{v} = \bm{e}^T\LRp{\bm{M}^k\LRp{\bm{M}^k_{c^2}}^{-1}\bm{M}^k- \bm{M}^k_{1/c^2}}\bm{e}  = -\alpha \LRp{\bm{v}^T\bm{e}}^2.
\]
Defining $\alpha = -{\rm sign}\LRp{\bm{v}^T\bm{e}}/\LRp{\bm{v}^T\bm{e}}$ then implies that the rank-one correction $\alpha \bm{v}\bm{v}^T$ is sufficient to enforce conservation.  Since $\LRp{\bm{v}^T\bm{e}}$ can be very small, $\alpha$ can be set to zero if  $\LRb{\bm{v}^T\bm{e}} \leq \delta \nor{\bm{v}}$ for some tolerance $\delta$ to ensure numerical stability.  The inverse of this conservative mass matrix can be applied using the Shermann-Morrison formula. Define $\tilde{\bm{v}} = \LRp{\bm{M}^k}^{-1}\bm{M}^k_{c^2}\LRp{\bm{M}^k}^{-1} \bm{v}$; assuming that ${1 + \alpha\bm{v}^T\tilde{\bm{v}}} \neq 0$,
\[
\LRp{\bm{M}^k\LRp{\bm{M}_{c^2}^k}^{-1}\bm{M}^k + \alpha \bm{v}\bm{v}^T}^{-1} = \LRp{\bm{M}^k}^{-1}\bm{M}^k_{c^2}\LRp{\bm{M}^k}^{-1} - \frac{\alpha\tilde{\bm{v}} \tilde{\bm{v}}^T}{1 + \alpha\bm{v}^T\tilde{\bm{v}}},
\]
requiring only $O(N^3)$ additional storage per element.  

For nonlinear hyperbolic problems with non-smooth solutions such as shocks, as a non-conservative scheme can lead to incorrect shock speeds \cite{leveque2002finite}.  The exact enforcement of local conservation is especially important in this context, since Theorem~\ref{lemma:cons} suggests that conservation errors depend otherwise on the regularity of $u$.

\section{Numerical examples}
\label{sec:num}
In this section, we give numerical examples confirming the estimates in Section~\ref{sec:estimates}, as well as numerical verification of convergence for the weight-adjusted DG method.  Numerical experiments are performed using a nodal DG method \cite{hesthaven2007nodal}; however, the weight-adjusted DG method is agnostic to the choice of basis used.  

\subsection{Comparisons between weighted and weight-adjusted inner products}

The DG method of Mercerat and Glinsky \cite{mercerat2015nodal} is energy stable with respect to the scaled $L^2$ norm induced by the inner product
\[
\int_{D^k} {w} p v + \bm{u}\cdot\bm{\tau} = \wip{p,v}{w} + \LRp{\bm{u},\bm{\tau}}_{L^2\LRp{D^k}},
\]
with $w = 1/c^2$.  The weight-adjusted DG method approximates this using the weight-adjusted inner product
\[
\int_{D^k} T_{1/w}^{-1}p v + \bm{u}\cdot\bm{\tau} = \waip{ p,v}{1/w} + \LRp{\bm{u},\bm{\tau}}_{L^2\LRp{D^k}}.
\]
We perform a numerical study to assess the quality of this approximation, which will influence how much the behavior of the weight-adjusted DG method will deviate from that of the standard DG method.  

Consider $u_{w,1}, u_{w,2}$ defined by the two scaled projection problems
\begin{align*}
\wip{u_{w,1},v}{w} &= \LRp{u,v}_{L^2\LRp{D^k}}, \qquad V \in V_h\LRp{D^k} \\
\waip{u_{w,2},v}{1/w} &= \LRp{u,v}_{L^2\LRp{D^k}}, \qquad V \in V_h\LRp{D^k}.
\end{align*}
$u_{w,1}, u_{w,2}$  approximate $u /w$.  If $u_{w,1}$ and $u_{w,2}$ are very close, the two projection problems are close to equivalent for that choice of $w$, and we expect the standard DG and weight-adjusted DG methods to behave similarly.  Polynomial expansion coefficients for $u_{w,1}, u_{w,2}$ are computed over each element by solving the matrix equations
\begin{align}
\bm{M}^k_{w}\bm{u}_{w,1} &= \bm{b} \label{eq:proj1}\\
\bm{M}^k \LRp{\bm{M}^k_{1/w}}^{-1}\bm{M}^k \bm{u}_{w,2} &= \bm{b}\label{eq:proj2},
\end{align}
where $\bm{b}_i = \int_{D^k} u \phi_i$.  We also examine convergence of $u_{w,3}$ to $uw$ as well, where coefficients for $u_{w,3}$ are computed by solving
\begin{align}
\LRp{\bm{M}^k \LRp{\bm{M}^k_{1/w}}^{-1}\bm{M}^k + \alpha \bm{v}\bm{v}^T}\bm{u}_{w,3} = \bm{b}. \label{eq:proj3}
\end{align}
Here, $\alpha$ and $\bm{v}$ define the rank-1 correction used to restore local conservation in Section~\ref{sec:conservation}.

\subsubsection{Regular solutions and weighting functions}
Table~\ref{table:projrates} shows $L^2$ errors for $\nor{{u_{w,1}} - u/{w}}_{L^2\LRp{\Omega}}$, $\nor{{u_{w,2}} - u/{w}}_{L^2\LRp{\Omega}}$, and $ \nor{{u_{w,3}} - u/w}_{L^2\LRp{\Omega}}$ on a sequence of uniform triangular meshes, with 
\[
u(x,y) = e^{x+y}, \qquad w(x,y) = 1 + \frac{1}{2}\sin(\pi x)\sin(\pi y).
\]
In all cases, the errors are very similar, though the error for $u_{w,1}$ (corresponding to the weighted $L^2$ inner product used in the standard DG method) appears to be consistently smaller than the errors for $u_{w,2}, u_{w,3}$.  Interestingly, the error for $u_{w,3}$, defined using the conservation-corrected mass matrix in (\ref{eq:proj3}), is smaller than the error for $u_{w,2}$ which does not include the rank-1 correction.

\begin{table}
\centering
\begin{tabular}{|c|c||c|c|c|c||c|}
\hline
&& $h = 1$ & $h = 1/2$, & $h = 1/4$ & $h = 1/8$ & Est.\ rate \\
\hline\hline
&$\nor{u_{w,1} - u/w}_{L^2}$ & 1.3920e-01 & 3.9460e-02 & 1.0207e-02 & 2.5739e-03 &1.922190 \\
$N = 1$&$\nor{u_{w,2} - u/w}_{L^2}$ & 1.4259e-01 & 3.9672e-02 & 1.0221e-02 & 2.5748e-03 &1.933027 \\
&$\nor{u_{w,3} - u/w}_{L^2}$ &1.4042e-01 & 3.9517e-02 & 1.0213e-02 & 2.5743e-03 &1.926034 \\
\hline\hline
&$\nor{u_{w,1} - u/w}_{L^2}$ & 3.1823e-02 & 4.5986e-03 & 5.9382e-04 & 7.4836e-05 &2.914944 \\
$N = 2$&$\nor{u_{w,2} - u/w}_{L^2}$ & 3.2454e-02 & 4.6209e-03 & 5.9455e-04 & 7.4859e-05 &2.923835 \\
&$\nor{u_{w,3} - u/w}_{L^2}$ & 3.2037e-02 & 4.6037e-03 & 5.9400e-04 & 7.4842e-05 &2.917925 \\
\hline\hline
   
&$\nor{u_{w,1} - u/w}_{L^2}$ & 6.2528e-03 & 4.0795e-04 & 2.5978e-05 & 1.6317e-06 &3.968489 \\
$N = 3$&$\nor{u_{w,2} - u/w}_{L^2}$ & 6.4703e-03 & 4.1129e-04 & 2.6034e-05 & 1.6326e-06 &3.983907 \\
&$\nor{u_{w,3} - u/w}_{L^2}$ & 6.2660e-03 & 4.0852e-04 & 2.5985e-05 & 1.6318e-06 &3.969530 \\

\hline\hline
&$\nor{u_{w,1} - u/w}_{L^2}$ & 7.9047e-04 & 2.8889e-05 & 9.3214e-07 & 2.9371e-08 &4.910195 \\
$N = 4$ &$\nor{u_{w,2} - u/w}_{L^2}$ & 7.9446e-04 & 2.8996e-05 & 9.3304e-07 & 2.9378e-08 &4.912661 \\
&$\nor{u_{w,3} - u/w}_{L^2}$ & 7.9433e-04 & 2.8902e-05 & 9.3226e-07 & 2.9377e-08 &4.912262 \\
\hline
\end{tabular}
\caption{$L^2$ errors and estimated rates of convergence for approximations $u_{w,1}, u_{w,2}, u_{w,3}$ of $u/w$ (defined by (\ref{eq:proj1}), (\ref{eq:proj2}), and (\ref{eq:proj3}) respectively) under uniform mesh refinement.  In this case, $u$ and $w$ are taken to be regular functions.}
\label{table:projrates}
\end{table}

\subsubsection{Solutions and weighting functions with decreased regularity}

It is worth noting that the results of Section~\ref{sec:estimates} involve terms $\nor{w}_{W^{N+1,\infty}},\nor{1/w}_{W^{N+1,\infty}}$ which depend on the regularity of $w$ over $D^k$.  For this reason, we expect the approximations $u_{w,1}, u_{w,2}, u_{w,3} \approx u/w$ resulting from the solutions of (\ref{eq:proj2}) and (\ref{eq:proj3}) to degenerate in quality as $w$ becomes less regular.  To test this, we take 
\[
c^2(x,y) = 1 + \sqrt{x^2 + y^2 + a}, \qquad a \in [0,\infty).
\]
which produces a non-differentiable cone as $a\rightarrow 0$.\footnote{Since typical quadratures are designed for more regular integrands, we double the quadrature strength when evaluating integrands with $a \approx 0$.  One-dimensional numerical experiments which compare increased quadrature strength with adaptive quadrature achieve qualitatively similar results.  Irregular weighting functions may also be dealt with using techniques from immersed DG methods \cite{adjerid2007higher}.}  Figure~\ref{fig:nonsmoothrates} shows the effect decreasing regularity of $w$ on the convergence of $u_{w,1}, u_{w,2}, u_{w,3}$ for $N = 3$.  While we do observe increases in error as $w$ loses regularity, we still observe that $u_{w,1}, u_{w,2}, u_{w,3}$ all behave very similarly independently of the regularity of $w$.  Along with the results of Theorem~\ref{lemma:mult}, this implies that the behavior of the weight-adjusted DG method should be very close to that of the standard DG method for both smooth and irregular $w$.  Interestingly, as $w$ approaches a non-differentiable function, the convergence of $u_{w,1}, u_{w,2}$, and $u_{w,3}$ to $u/w$ reduces to $O(h^2)$ for all orders $N$ tested.

\begin{figure}
\centering
\subfloat[$N=3$]{
\begin{tikzpicture}
\begin{loglogaxis}[
	legend cell align=left,
	width=.475\textwidth,
    xlabel={Mesh size $h$},
    ylabel={$L^2$ error},
    xmin=.01, xmax=1.5,
    ymin=1e-10, ymax=5e-3,
    legend pos=south east,
    xmajorgrids=true,
    ymajorgrids=true,
    grid style=dashed,
] 
\addplot[color=magenta,mark=*,semithick, mark options={fill=markercolor}]
coordinates{(0.5,0.000811841)(0.25,5.2479e-05)(0.125,3.34793e-06)(0.0625,2.10171e-07)};
\addplot[color=magenta,mark=square*,semithick, mark options={fill=markercolor}]
coordinates{(0.5,0.000815045)(0.25,5.25257e-05)(0.125,3.34867e-06)(0.0625,2.10183e-07)};
\addplot[color=magenta,mark=x,semithick, mark options={fill=markercolor}]
coordinates{(0.5,0.000813198)(0.25,5.25014e-05)(0.125,3.34828e-06)(0.0625,2.10177e-07)};
\node at (axis cs:.03,2.1e-07) {$a = 10^{-1}$};

\addplot[color=black,mark=*,semithick, mark options={fill=markercolor}]
coordinates{(0.5,0.00118692)(0.25,0.000190298)(0.125,1.821e-05)(0.0625,1.27298e-06)};
\addplot[color=black,mark=square*,semithick, mark options={fill=markercolor}]
coordinates{(0.5,0.00119612)(0.25,0.000190761)(0.125,1.82156e-05)(0.0625,1.27306e-06)};
\addplot[color=black,mark=x,semithick, mark options={fill=markercolor}]
coordinates{(0.5,0.00118848)(0.25,0.000190313)(0.125,1.82102e-05)(0.0625,1.27299e-06)};
\node at (axis cs:.03,1.27e-06) {$a = 10^{-2}$};

\addplot[color=red,mark=*,semithick, mark options={fill=markercolor}]
coordinates{(0.5,0.00123752)(0.25,0.000276383)(0.125,6.84113e-05)(0.0625,1.03818e-05)};
\addplot[color=red,mark=square*,semithick, mark options={fill=markercolor}]
coordinates{(0.5,0.00124408)(0.25,0.00027695)(0.125,6.84631e-05)(0.0625,1.03834e-05)};
\addplot[color=red,mark=x,semithick, mark options={fill=markercolor}]
coordinates{(0.5,0.00123967)(0.25,0.000276403)(0.125,6.84115e-05)(0.0625,1.03818e-05)};
\node at (axis cs:.03,.75e-05) {$a = 10^{-3}$};

\addplot[color=blue,mark=*,semithick, mark options={fill=markercolor}]
coordinates{(0.5,0.0012956)(0.25,0.000301894)(0.125,7.66855e-05)(0.0625,1.9284e-05)};
\addplot[color=blue,mark=square*,semithick, mark options={fill=markercolor}]
coordinates{(0.5,0.0013018)(0.25,0.000302283)(0.125,7.67179e-05)(0.0625,1.92874e-05)};
\addplot[color=blue,mark=x,semithick, mark options={fill=markercolor}]
coordinates{(0.5,0.00129769)(0.25,0.000301911)(0.125,7.66858e-05)(0.0625,1.92853e-05)};
\node at (axis cs:.03,2.25e-05) {$a = 10^{-4}$};

\legend{$u_{w,1}$, $u_{w,2}$, $u_{w,3}$}
\end{loglogaxis}
\end{tikzpicture}
}
\subfloat[$N=4$]{
\begin{tikzpicture}
\begin{loglogaxis}[
	legend cell align=left,
	width=.475\textwidth,
    xlabel={Mesh size $h$},
    ylabel={$L^2$ error},
    xmin=.01, xmax=1.5,
    ymin=1e-10, ymax=5e-3,
    legend pos=south east,
    xmajorgrids=true,
    ymajorgrids=true,
    grid style=dashed,
] 
\addplot[color=magenta,mark=*,semithick, mark options={fill=markercolor}]
coordinates{(0.5,7.68472e-05)(0.25,3.6451e-06)(0.125,1.25564e-07)(0.0625,4.00732e-09)};
\addplot[color=magenta,mark=square*,semithick, mark options={fill=markercolor}]
coordinates{(0.5,7.71628e-05)(0.25,3.64757e-06)(0.125,1.25583e-07)(0.0625,4.00747e-09)};
\addplot[color=magenta,mark=x,semithick, mark options={fill=markercolor}]
coordinates{(0.5,7.69985e-05)(0.25,3.6457e-06)(0.125,1.25573e-07)(0.0625,4.00743e-09)};
\node at (axis cs:.03,4.e-09) {$a = 10^{-1}$};

\addplot[color=black,mark=*,semithick, mark options={fill=markercolor}]
coordinates{(0.5,0.000379655)(0.25,4.1344e-05)(0.125,2.53271e-06)(0.0625,1.42708e-07)};
\addplot[color=black,mark=square*,semithick, mark options={fill=markercolor}]
coordinates{(0.5,0.00038317)(0.25,4.1436e-05)(0.125,2.53325e-06)(0.0625,1.42717e-07)};
\addplot[color=black,mark=x,semithick, mark options={fill=markercolor}]
coordinates{(0.5,0.000379681)(0.25,4.13447e-05)(0.125,2.53276e-06)(0.0625,1.42708e-07)};
\node at (axis cs:.03,1.42e-07) {$a = 10^{-2}$};

\addplot[color=red,mark=*,semithick, mark options={fill=markercolor}]
coordinates{(0.5,0.000461945)(0.25,0.000126468)(0.125,2.25822e-05)(0.0625,1.63882e-06)};
\addplot[color=red,mark=square*,semithick, mark options={fill=markercolor}]
coordinates{(0.5,0.000465012)(0.25,0.000126831)(0.125,2.26025e-05)(0.0625,1.639e-06)};
\addplot[color=red,mark=x,semithick, mark options={fill=markercolor}]
coordinates{(0.5,0.000462034)(0.25,0.000126466)(0.125,2.25821e-05)(0.0625,1.63882e-06)};
\node at (axis cs:.03,1.5e-06) {$a = 10^{-3}$};

\addplot[color=blue,mark=*,semithick, mark options={fill=markercolor}]
coordinates{(0.5,0.000483624)(0.25,0.000126702)(0.125,3.34245e-05)(0.0625,8.06305e-06)};
\addplot[color=blue,mark=square*,semithick, mark options={fill=markercolor}]
coordinates{(0.5,0.000485939)(0.25,0.000126923)(0.125,3.34475e-05)(0.0625,8.06499e-06)};
\addplot[color=blue,mark=x,semithick, mark options={fill=markercolor}]
coordinates{(0.5,0.000483704)(0.25,0.000126703)(0.125,3.34244e-05)(0.0625,8.06305e-06)};
\node at (axis cs:.03,8.5e-06) {$a = 10^{-4}$};

\legend{$u_{w,1}$, $u_{w,2}$, $u_{w,3}$}
\end{loglogaxis}
\end{tikzpicture}
}
\caption{Convergence of $L^2$ errors for solutions $u_{w,1}, u_{w,2}, u_{w,3}$ of (\ref{eq:proj1}), (\ref{eq:proj2}), (\ref{eq:proj3}) under uniform mesh refinement for $N = 3,4$.  In this case, $w$ is taken to be a function whose regularity decreases as $a\rightarrow 0$.   }
\label{fig:nonsmoothrates}
\end{figure}
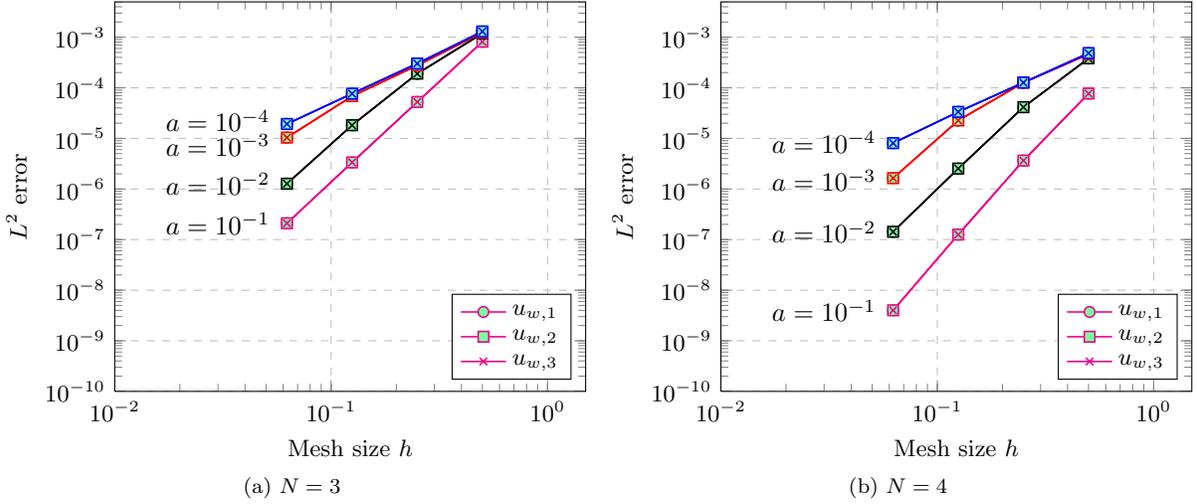


\subsection{Local conservation errors}

Section~\ref{sec:conservation} discusses the fact that the weight-adjusted DG method does not locally conserve the same quantities conserved by the standard DG method.  However, estimates show that for sufficiently regular $u$ and $w$, the conservation error converges at $O(h^{2N+2})$.  

\subsubsection{Regular solutions and weighting functions}

We test this first for regular $u,w$ by taking  
\[
u(x,y) = e^{x+y}, \qquad w(x,y) = 1 + \frac{1}{2}\sin(\pi x)\sin(\pi y).
\]
and computing the conservation errors for $u_{w,2}, u_{w,3}$.  For $u_{w,2}$, this error is defined as 
\[
\sum_k \LRp{\int_{D^k} \frac{u_{w,1}}{c^2} - \int_{D^k}T_{1/w}^{-1} u_{w,2}}, 
\]
for $u_{w,1}, u_{w,2}$ as defined in (\ref{eq:proj1}),(\ref{eq:proj2}), and (\ref{eq:proj3}), respectively.  For $u_{w,3}$ since the conservation-corrected mass matrix does not have a clear inner product analogue, we measure the conservation error via
\[
\sum_k \bm{e}^T\bm{M}_{1/c^2}^k\bm{u}_{w,1} - \bm{e}^T\bm{M}^k\LRp{\bm{M}^{k}_{c^2}}^{-1} \bm{M}^k\bm{u}_{w,3},
\]
where $\bm{e}$ are the polynomial expansion coefficients for the constant $1$ over $D^k$.  

In all experiments, $\alpha$ is set to zero if $\LRb{\bm{v}^T\bm{e}} \leq \delta \nor{\bm{v}}$ for $\delta = 10^{-8}$.  Table~\ref{table:smoothuw} shows the conservation errors
\[
\LRb{\overline{u_{w,1}/c^2} - \overline{T^{-1}_{c^2}u_{w,2}}}, \qquad \LRb{\overline{u_{w,1}/c^2} - \overline{T^{-1}_{c^2}u_{w,3}}}
\]
for $u_{w,2}$ and $u_{w,3}$ respectively.  The estimated rate of convergence for $u_{w,2}$ is also reported.  As predicted in Section~\ref{sec:conservation}, the conservation error for $u_{w,2}$ is observed to converge at a rate of $O(h^{2N+2})$, while $u_{w,3}$ is observed to reduce conservation error to machine precision values.  

\begin{table}
\centering
\begin{tabular}{|c|c||c|c|c|c||c|}
\hline
&& $h = 1$ & $h = 1/2$, & $h = 1/4$ & $h = 1/8$ & Est.\ rate \\
\hline
$N = 1$& $\LRb{\overline{u_{w,1}/c^2}-\overline{u_{w,2}/c^2}}$ & 9.5935e-03 & 7.9155e-04 & 5.2323e-05 & 3.2990e-06 &3.953251 \\
$N = 1$& $\LRb{\overline{u_{w,1}/c^2}-\overline{u_{w,3}/c^2}}$ & 2.7409e-16 & 2.7712e-16 & 2.5468e-16 & 2.5320e-16 & \\
\hline\hline
$N = 2$& $\LRb{\overline{u_{w,1}/c^2}-\overline{u_{w,2}/c^2}}$  & 4.4236e-04 & 1.4430e-05 & 2.3578e-07 & 3.7821e-09 &5.948822 \\
$N = 2$& $\LRb{\overline{u_{w,1}/c^2}-\overline{u_{w,3}/c^2}}$  & 2.9046e-16 & 3.1423e-16 & 3.3770e-16 & 3.4679e-16 & \\
\hline\hline
$N = 3$& $\LRb{\overline{u_{w,1}/c^2}-\overline{u_{w,2}/c^2}}$  & 7.7600e-05 & 3.5645e-07 & 1.5276e-09 & 6.2161e-12 &7.903656 \\
$N = 3$& $\LRb{\overline{u_{w,1}/c^2}-\overline{u_{w,3}/c^2}}$  & 3.6527e-16 & 2.9679e-16 & 3.5446e-16 & 3.5605e-16 & \\
\hline\hline
$N = 4$& $\LRb{\overline{u_{w,1}/c^2}-\overline{u_{w,2}/c^2}}$  & 2.5627e-06 & 7.8864e-09 & 1.2094e-11 & 1.3714e-14 &9.566707 \\
$N = 4$& $\LRb{\overline{u_{w,1}/c^2}-\overline{u_{w,3}/c^2}}$  & 3.2904e-16 & 2.9661e-16 & 3.2352e-16 & 3.3249e-16 & \\
\hline
\end{tabular}
\caption{Conservation errors at different orders of approximation $N$ under uniform mesh refinement for solutions $u_{w,2}, u_{w,3}$ to (\ref{eq:proj2}), (\ref{eq:proj3}).  In this case, $u,w$ are taken to be regular functions. Estimated orders of convergence are also reported for $u_{w,2}$.  }
\label{table:smoothuw}
\end{table}

\subsubsection{Solutions and weighting functions with decreased regularity}
We also investigate how the regularity of $u,w$ affect local conservation errors.  We consider $u,w$ given both a by smooth exponential and a regularized cone
\begin{align*}
u(x,y) = e^{x+y}, \qquad w(x,y) = 1 + \sqrt{x^2 + y^2 + a}, \quad a \in [0,\infty),\\
u(x,y) = 1 + \sqrt{x^2 + y^2 + a}, \quad a \in [0,\infty), \qquad w(x,y) = e^{x+y}.
\end{align*}
Figure~\ref{fig:conserrregularity} shows the effects of decreasing regularity of $w$ and $u$ separately on the conservation errors.  Decreasing regularity of $w$ is observed to reduce convergence of conservation errors to $O(h^4)$. Interestingly, only decreasing the regularity of $u$ affects conservation errors far less than only decreasing the regularity of $w$, suggesting that the bound in Theorem~\ref{lemma:cons} may not be sharp.  Additionally, for less regular $u$ and discontinuous $u$, we observe numerically that conservation errors decrease at a rate of $O(h^{N+2})$.  Both of these behaviors are better than expected from Theorem~\ref{lemma:cons}, and suggest that conservation errors do not depend strongly on the regularity of $u$.  

\begin{figure}
\centering
\subfloat[Conservation errors for less-regular $w$]{
\begin{tikzpicture}
\begin{loglogaxis}[
	legend cell align=left,
	width=.49\textwidth,
    xlabel={Mesh size $h$},
    xmin=.01, xmax=1.5,
    ymin=1e-15, ymax=1e-5,
    legend pos=south east,
    xmajorgrids=true,
    ymajorgrids=true,
    grid style=dashed,
] 

\addplot[color=blue,mark=*,semithick, mark options={fill=markercolor}]
coordinates{(0.5,1.65584e-07)(0.25,8.21601e-10)(0.125,3.86473e-12)(0.0625,1.61111e-14)};
\node at (axis cs:.03,1.6e-14) {$a = 10^{-1}$};

\addplot[color=red,mark=square*,semithick, mark options={fill=markercolor}]
coordinates{(0.5,1.37583e-06)(0.25,5.34628e-08)(0.125,4.48005e-10)(0.0625,2.10452e-12)};
\node at (axis cs:.03,2.1e-12) {$a = 10^{-2}$};

\addplot[color=black,mark=triangle*,semithick, mark options={fill=markercolor}]
coordinates{(0.5,1.65946e-06)(0.25,1.0755e-07)(0.125,5.90813e-09)(0.0625,1.26104e-10)};
\node at (axis cs:.03,1e-10) {$a = 10^{-3}$};

\addplot[color=magenta,mark=diamond*,semithick, mark options={fill=markercolor}]
coordinates{(0.5,1.92368e-06)(0.25,1.24548e-07)(0.125,7.16551e-09)(0.0625,4.17446e-10)};
\node at (axis cs:.03,4.6e-10) {$a = 10^{-4}$};

\end{loglogaxis}
\end{tikzpicture}
}
\subfloat[Conservation errors for less-regular $u$]{
\begin{tikzpicture}
\begin{loglogaxis}[
	legend cell align=left,
	width=.49\textwidth,
    xlabel={Mesh size $h$},
    xmin=.01, xmax=1.5,
    ymin=1e-15, ymax=1e-5,
    legend pos=south east,
    xmajorgrids=true,
    ymajorgrids=true,
    grid style=dashed,
] 

\addplot[color=blue,mark=*,semithick, mark options={fill=markercolor}]
coordinates{(0.5,1.93878e-07)(0.25,7.97778e-10)(0.125,3.16898e-12)(0.0625,1.34386e-14)};
\node at (axis cs:.03,.5e-14) {$a = 10^{-1}$};

\addplot[color=red,mark=square*,semithick, mark options={fill=markercolor}]
coordinates{(0.5,2.12529e-07)(0.25,1.28147e-09)(0.125,5.42966e-12)(0.0625,2.17705e-14)};
\node at (axis cs:.03,1.5e-14) {$a = 10^{-2}$};

\addplot[color=black,mark=triangle*,semithick, mark options={fill=markercolor}]
coordinates{(0.5,1.91045e-07)(0.25,1.29447e-09)(0.125,8.90727e-12)(0.0625,4.98508e-14)};
\node at (axis cs:.03,4.5e-14) {$a = 10^{-3}$};

\addplot[color=magenta,mark=diamond*,semithick, mark options={fill=markercolor}]
coordinates{(0.5,1.84712e-07)(0.25,1.16594e-09)(0.125,7.75168e-12)(0.0625,5.98593e-14)};
\node at (axis cs:.03,1.5e-13) {$a = 10^{-4}$};


\end{loglogaxis}
\end{tikzpicture}
}
\caption{Convergence of conservation errors for solution $u_{w,2}$ to (\ref{eq:proj2}) under uniform mesh refinement.  In this case, $u$, $w$ are taken to be functions whose regularity decreases as $a\rightarrow 0$.  Results are shown for $N = 3$. }
\label{fig:conserrregularity}
\end{figure}
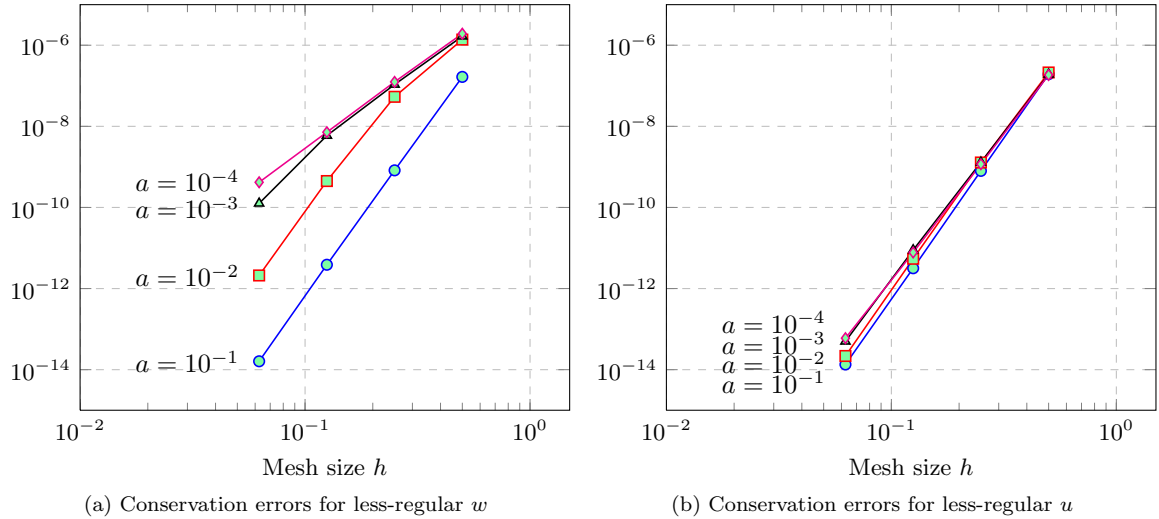

\subsection{Convergence of DG for heterogeneous wavespeed}

In this section, we examine the convergence of high order standard and weight-adjusted DG methods to manufactured and reference solutions under a wavespeed which varies spatially with each element.  

\subsubsection{Convergence to a manufactured solution}
For the acoustic wave equation with smoothly varying wavespeed, there are few analytic reference solutions in higher dimensions.  For this reason the method of manufactured solutions is often used to analyze the convergence of methods for wave propagation in heterogeneous media \cite{castro2010seismic, mercerat2015nodal}.  The method of manufactured solutions chooses expressions for $p, \bm{u}$ and determines a source term $f\LRp{\bm{x},t}$ such that the inhomogeneous acoustic wave equations
\begin{align}
\frac{1}{c^2}\pd{p}{t}{} + \Div u &= f \nonumber\\
\rho\pd{\bm{u}}{t}{} + \Grad p &= 0,
\label{eq:dgconv}
\end{align}
have solution $p,\bm{u}$.  Table~\ref{table:manusol} shows the convergence of $L^2$ errors for both standard DG and weight-adjusted DG on a sequence of 2D uniform triangular meshes for
\[
c^2(x,y) = 1 + \frac{1}{2}\sin\LRp{\pi x} \sin\LRp{\pi y}, \qquad p(x,y,t) = \cos\LRp{\frac{\pi}{2} x}\cos\LRp{\frac{\pi}{2} y}\cos\LRp{\frac{\pi}{2}\sqrt{2}t}.
\] 
A triangular quadrature from Xiao and Gimbutas \cite{xiao2010quadrature} (chosen to be exact for polynomials up to degree $3N$) is used to compute both the weighted and weight-adjusted mass matrices for standard DG and the application of the weighted-adjusted mass matrix for weight-adjusted DG.  We do not correct the mass matrix with $\alpha \bm{v}\bm{v}^T$ to enforce local conservation in the following numerical experiments.

\begin{table}
\centering                                                                                   
\subfloat[Standard DG $L^2$ errors, manufactured solution]{
\begin{tabular}{|c||c|c|c|c|}
\hline                                                   
$N$ & $h = 1$ & $h = 1/2$ & $h = 1/4$ & $h = 1/8$\\ 
\hline                                                   
$ 1$ &   2.13e-01 & 6.25e-02 & 1.64e-02 & 4.19e-03 \\
\hline                                                   
$ 2$ &  3.01e-02 & 3.60e-03 & 4.21e-04 & 5.07e-05 \\
\hline                                                   
$ 3$ &   6.10e-03 & 3.33e-04 & 2.04e-05 & 1.22e-06 \\
\hline                                                   
$ 4$ &  6.61e-04 & 2.12e-05 & 6.39e-07 & 1.94e-08 \\
\hline 
\end{tabular}
}
\subfloat[Weight-adjusted DG $L^2$ errors, manufactured solution]{
\begin{tabular}{|c||c|c|c|c|}
\hline                                                   
$N$ & $h = 1$ & $h = 1/2$ & $h = 1/4$ & $h = 1/8$\\ 
\hline                                                   
$ 1$ &  2.05e-01 & 5.99e-02 & 1.62e-02 & 4.18e-03 \\
\hline                                                   
$ 2$ &  2.89e-02 & 3.54e-03 & 4.18e-04 & 5.07e-05 \\
\hline                                                   
$ 3$ &   8.69e-03 & 3.47e-04 & 2.03e-05 & 1.22e-06 \\
\hline                                                   
$ 4$ & 1.09e-03 & 2.27e-05 & 6.30e-07 & 1.93e-08 \\
\hline 
\end{tabular}
}\\

\subfloat[Standard DG $L^2$ errors, reference solution]{
\begin{tabular}{|c||c|c|c|c|}
\hline                                                   
$N$ & $h = 1$ & $h = 1/2$ & $h = 1/4$ & $h = 1/8$\\ 
\hline                                                   
$ 1$ &   2.48e-01 & 7.58-02 & 1.69e-02 & 4.46e-03 \\
\hline                                                   
$ 2$ &  5.95e-02 & 9.95e-03 & 1.10e-03 & 1.22e-04 \\
\hline                                                   
$ 3$ &   2.29e-02 & 1.98e-03 & 9.52e-05 & 6.56e-06 \\
\hline                                                   
$ 4$ &  4.90e-03 & 3.01e-04 & 1.78e-05 & 7.27e-07 \\
\hline 
\end{tabular}
}
\subfloat[Weight-adjusted DG $L^2$ errors, reference solution]{
\begin{tabular}{|c||c|c|c|c|}
\hline                                                   
$N$ & $h = 1$ & $h = 1/2$ & $h = 1/4$ & $h = 1/8$\\ 
\hline                                                
$ 1$ &  2.50e-01 & 7.72e-02 & 1.69e-02 & 4.47e-03 \\
\hline                                                   
$ 2$ &  6.09e-02 & 1.02e-02 & 1.10e-03 & 1.22e-04 \\
\hline                                                   
$ 3$ &   1.98e-02 & 1.98e-03 & 9.52e-05 & 6.56e-06 \\
\hline                                                   
$ 4$ & 4.64e-03 & 3.02e-04 & 1.78e-05 & 7.28e-07 \\
\hline 
\end{tabular}
}

\caption{Convergence of $L^2$ errors for standard and weight-adjusted DG solutions to (\ref{eq:dgconv}) for manufactured and reference solutions at time $T = 1$ under uniform triangular mesh refinement.}
\label{table:manusol}
\end{table}

\subsubsection{Convergence to a reference solution}

We also compare the convergence of DG for heterogeneous media in a more realistic setting by computing the error with respect to a fine-grid reference solution computed using a spectral method over the bi-unit square $[-1,1]^2$ with $N = 100$.  The timestep for the reference solution is taken sufficiently small as to make temporal errors negligible.  The same wavespeed $c$ used for the manufactured solution is used again for the manufactured solution, with an initial condition of $p(x,y,0) = \cos\LRp{\frac{\pi}{2} x}\cos\LRp{\frac{\pi}{2} y}$.  Table~\ref{table:rates} shows estimated rates of convergence for both standard and weight-adjusted DG.  For both methods, rates of convergence between  $O(h^{N+1/2})$ and  $O(h^{N+1})$ are observed for $N = 1,\ldots,4$. In all cases, the errors for the standard and weight-adjusted DG methods are nearly identical for on all but the coarsest mesh.

\begin{table}
\centering
\subfloat[Rates of convergence to manufactured solution]{
\begin{tabular}{|c||c|c|c|c|}
\hline                                                   
 & $N = 1$ & $N = 2$ & $N = 3$ & $N = 4$\\ 
\hline                                                   
DG&       1.9220  &     3.0752     &      4.0440      &      5.0446    \\
\hline                                                   
WADG &  1.9211 &     3.0629     &    4.0752    &   5.0990\\
\hline 
\end{tabular}
}
\subfloat[Rates of convergence to reference solution]{
\begin{tabular}{|c||c|c|c|c|}
\hline                                                   
 & $N = 1$ & $N = 2$ & $N = 3$ & $N = 4$\\ 
\hline                                                   
DG &   1.8256  &      3.1796  &    3.8589    &  4.6171         \\
\hline                                                   
WADG &  1.8425 &      3.1807  & 3.8583   & 4.6128 \\
\hline 
\end{tabular}
}
\caption{Estimated rates of convergence of standard and weight-adjusted DG solutions of (\ref{eq:dgconv}) to both manufactured and reference solutions at $T=1$.}
\label{table:rates}
\end{table}

Finally, Figure~\ref{fig:snap} shows a comparison of the standard and weight-adjusted DG method for the discontinuous wavespeed
\begin{align}
c^2(x,y) = \begin{cases}
1 + \frac{1}{2}\sin\LRp{2\pi x}\sin\LRp{2\pi y}, \qquad y \leq 0\\
2 + \frac{1}{2}\sin\LRp{2\pi x}\sin\LRp{2\pi y}, \qquad y > 0.
\end{cases}
\label{eq:cdisc}
\end{align}
The initial condition is taken to be a initial Gaussian pulse centered at $\LRp{0,1/4}$.  For $N = 4$, $h = 1/8$, and $T = .5$, both the standard DG and weight-adjusted DG solutions are indistinguishable.

\begin{figure}
\centering
\subfloat[Standard DG]{
\includegraphics[width=.375\textwidth]{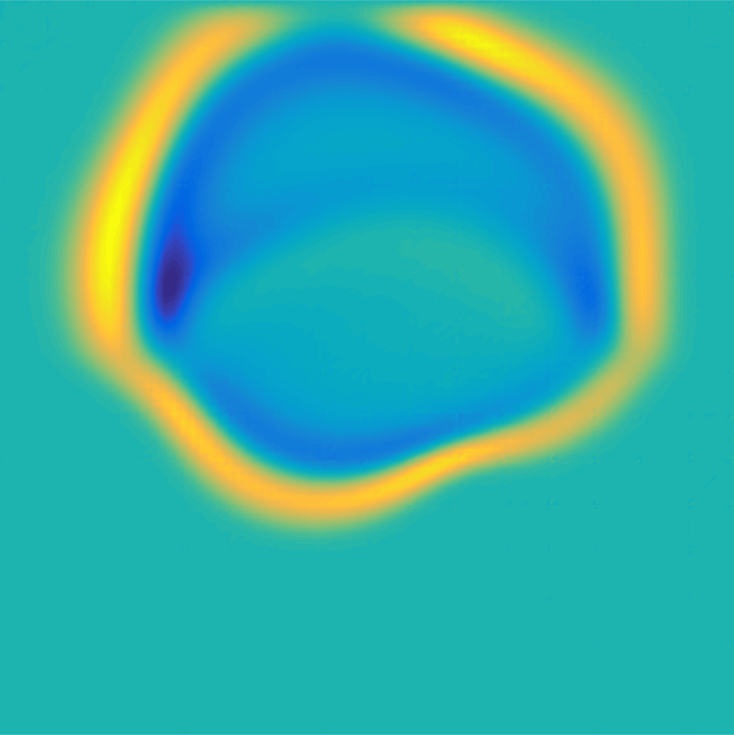}
}
\vspace{1em}
\subfloat[Weight-adjusted DG]{
\includegraphics[width=.375\textwidth]{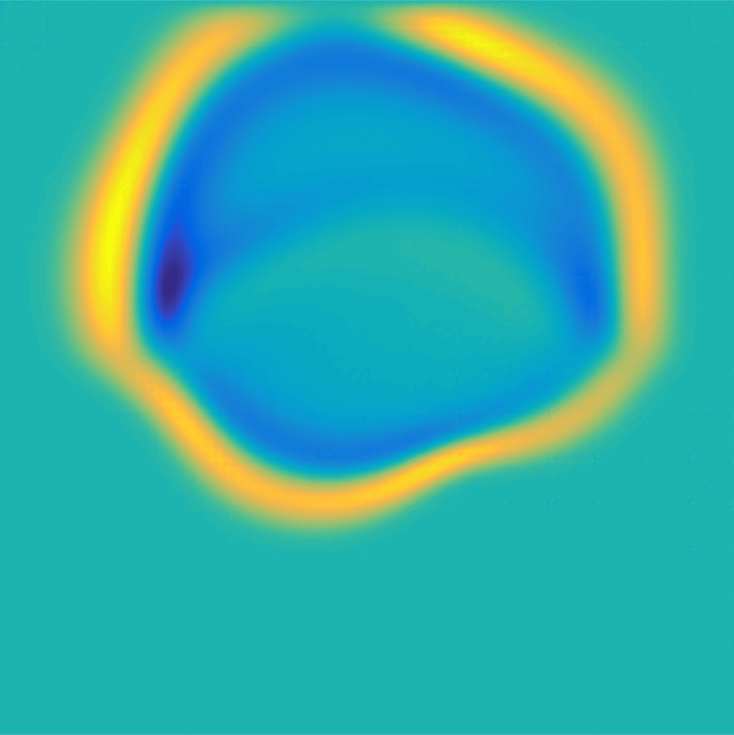}
}
\caption{Snapshot of from standard and weight-adjusted DG solutions of the acoustic wave equation with $c^2$ defined by (\ref{eq:cdisc}).  The order of approximation is $N = 4$, and the final time is taken to be $T=.5$.  The initial condition is a Gaussian pulse centered around $(0,.25)$, and $c^2$ varies spatially with a discontinuity at $y = 0$. }
\label{fig:snap}
\end{figure}

\subsection{Effect of reduced quadrature}

It was noted in \cite{warburton2013low} that, for the LSC-DG formulation, it is possible to reduce the order of the quadrature used to evaluate the variational formulation significantly without compromising the estimated order of convergence implied by theory.  This can be attributed to two facts: first, that stability of the LSC-DG formulation does not depend on quadrature strength, and secondly, that errors for a degree $2N$ quadrature rule are of the same order as the discretization error.  

Similarly, the weight-adjusted DG method is energy stable so long as the weight-adjusted inner product (computed using quadrature) induces a norm.  Numerical experiments indicate that quadrature degrees which integrate degree $2N+1$ polynomials exactly rule are sufficient, and that increasing quadrature strength beyond this degree does not offer any significant advantages.  Table~\ref{table:quad} shows the effect of varying the quadrature strength $N_q$ from degree $2N-1$ to $3N$ for an $N = 4$ discretization.  While the error decreases very slightly by increasing the degree of quadrature from $2N-1$ to $2N$ or $2N+1$, no significant change in error is observed by increasing the degree of quadrature beyond than $2N+1$.  Results are not reported for quadratures of lower degree than $2N-1$, as numerically singular mass matrices are generated.

\begin{table}
\centering
\subfloat[Manufactured solution]{
\begin{tabular}{|c||c|c|}
\hline
$N_q$ & $L^2$ error (DG) & $L^2$ error (WADG) \\
\hline
7 &   1.0102e-07  & 2.9122e-08\\ 
\hline
8 &   2.1710e-08 & 2.1709e-08\\
\hline
9 & 1.9548e-08  &1.9544e-08\\
\hline
10 & 1.9443e-08 & 1.9544e-08\\
\hline
11 & 1.9443e-08  &1.9324e-08\\
\hline
12 & 1.9443e-08  & 1.9324e-08\\
\hline
\end{tabular}
}
\subfloat[Reference solution]{
\begin{tabular}{|c||c|c|}
\hline
$N_q$ & $L^2$ error (DG) & $L^2$ error (WADG) \\
\hline
7 &      7.7932e-07 & 8.3296e-07\\
\hline
8 &      7.6739e-07 &  7.6732e-07 \\
\hline
9 &     7.6568e-07 & 7.6553e-07 \\
\hline
10 &    7.6504e-07 &  7.6410e-07\\
\hline
11 &    7.6410e-07 & 7.6502e-07\\
\hline
12 &      7.6501e-07 & 7.6412e-07 \\
\hline
\end{tabular}
}
\caption{Effect of varying quadrature degree from $2N-1$ to $3N$ on $L^2$ errors for the standard and weight-adjusted DG solution of (\ref{eq:dgconv}).  Results are for $N = 4$ and a uniform $h = 1/8$ mesh.}
\label{table:quad}
\end{table}

\section{Conclusions and future work}

This work introduces a weight-adjusted DG (WADG) method for the simulation of wave propagation in heterogeneous media which is both provably energy stable and high order accurate for heterogeneous media with wavespeeds which are locally smooth over each element.  Additionally, the implementation of the WADG method is non-invasive, and can be incorporated into a DG code for wave propagation in isotropic media with only minor modifications.  

The WADG method relies on an approximation of the weighted mass matrix by an equivalent weight-adjusted mass matrix, which implies that unlike the DG method, the method is no longer Galerkin consistent or locally conservative (for non-polynomial wavespeeds).  However, the method is shown to be asymptotically consistent and high order accurate, while conservation errors are shown to superconverge at rate $O(h^{2N+2})$ for smooth solutions and wavespeeds.  Finally, numerical experiments also indicate that a low-rank correction to the mass matrix can be used to recover exact conservation properties in the case of non-polynomial wavespeed.  

Future work will involve the efficient implementation of the WADG method on GPUs for more realistic velocity models in three dimensions, as well as the extension of the WADG method to curvilinear meshes, which can be used to the control interface errors resulting from the approximation of non-planar interfaces by piecewise planar surfaces \cite{wang2009discontinuous}.  We note that while the implementation of the WADG method for curvilinear meshes is relatively similar, the analysis differs from the case of affine elements.  

\section{Acknowledgments} 

The authors thank TOTAL for permission to publish.  JC and TW are funded by a grant from TOTAL E\&P Research and Technology USA.  


\bibliographystyle{unsrt}
\bibliography{total}{}

\end{document}